\documentclass[12pt]{article}

\usepackage{amssymb,amsmath, amsthm}
\usepackage{verbatim,mathrsfs,bbm}
\usepackage{bbold,mathbbol}
\usepackage[colorlinks,citecolor=red]{hyperref}
\usepackage{color}
\usepackage{young}
\usepackage[english]{babel}
\usepackage{tikz}
\usetikzlibrary{matrix}

\newtheorem{Theorem}{Theorem}
\newtheorem{Lemma}[Theorem]{Lemma}
\newtheorem{Example}[Theorem]{Example}

\newtheorem{Proposition}[Theorem]{Proposition}

\newcommand{\N}{\mathbb{N}}

\newcommand{\Z}{\mathbb{Z}}

\newcommand{\wil}[1]{\widetilde{#1}}

\newcommand{\ds}[1]{\displaystyle{#1}}

\usepackage[all]{xy}

\begin{document}

\title{Lexicographic Shellability of Partial Involutions}
\date{\today}
\author{Mahir Bilen Can, Tim Twelbeck}
\maketitle

\begin{abstract}
In this manuscript we study inclusion posets of Borel orbit closures on (symmetric) matrices.
In particular, we show that the Bruhat poset of partial involutions is a lexicographically shellable poset. 
Also, studying the embeddings of symmetric groups and involutions into rooks and partial involutions, respectively, 
we find new $EL$-labelings on permutations as well as on involutions.  
\end{abstract}

\section{\textbf{Introduction.}}

Recall that a simplicial complex $\varDelta$ is called {\em shellable} if there exits a linear ordering 
$F_1,F_2,\dots, F_k$ of the facets of $\varDelta$ in such a way that, for each $j=2,\dots, k$, 
the intersection of the sub-complex of $F_j$ with the union of all sub-complexes of previous facets 
$F_1,\dots,F_{j-1}$ is a pure sub-complex of $\varDelta$ of dimension $\dim F_j -1$.

Although its definition is not illuminating, the notion of shellability has remarkable topological consequences. 
For example, if shellable, the simplicial complex $\varDelta$ has the homotopy type of a wedge of spheres. 
See \cite{BW96}. One of our purposes in this manuscript is to prove the shellability of a certain simplicial complex 
arising from an action of the invertible upper triangular matrices on the space of symmetric matrices.

Let $K$ denote an algebraically closed field of characteristic zero, $M=M_n(K)$ denote the affine variety of 
$n\times n$ matrices over $K$, and let $Y\subseteq M$ be a subvariety on which a group $G$ acts. 
Let $B(Y;G)$ denote the set Zariski-closures of the $G$-orbits in $M$. 
We focus on the following examples:
\begin{itemize}
\item $Y=Q$, the space of symmetric matrices in $M$, and $G=B_n$, the 
Borel group of invertible upper triangular matrices acting on $Y$ via 
\begin{align*}
x \cdot A = (x^{-1})^\top A x^{-1}, 
\end{align*}
where $x^\top$ denotes the transpose of the matrix $x \in B_n$ and $A\in Q$. 
\item $Y=M$ and $G=B_n\times B_n$ acting on $Y$ via 
\begin{align}\label{GL_n acts on M_n}
(x,y)\cdot A = x A y^{-1},
\end{align}
where $x,y\in B_n$ and $A\in M$.
\end{itemize}
In both of these examples, $B(Y;G)$ is finite, and furthermore, it is partially ordered with respect to 
set inclusion.

It is well known that the set of all chains in a poset forms a simplicial complex, 
called the {\em order complex} of the poset.
Let $\varDelta(Y)$ denote the order complex of $(B(Y;G),\subseteq)$. 
Our first main result is that $\varDelta(Q)$ is a shellable complex.
In fact, we prove a much stronger statement;
the poset $(B(Q;B_n),\subseteq)$ is ``lexicographically shellable.'' 
Introduced by Bj\"orner in \cite{Bjorner80} and advanced by Bj\"orner and Wachs in \cite{BW82}, 
the notion of lexicographic shellability amounts to finding a suitable labeling of the edges of the 
Hasse diagram of the poset under consideration. Thus, associated with each saturated chain 
is a sequence of labels, which provides an ordering of the faces of the simplicial complex $\varDelta(Y)$.

The inclusion orders which are defined above have concrete descriptions. 
Recall that the {\em rook monoid} $R_n$ is the finite monoid of 0/1 matrices with at most one 1 in each 
row and each column. It is well known that the elements of $R_n$ parametrize the orbits of the action
(\ref{GL_n acts on M_n}) of $B_n\times B_n$ on $M$. See \cite{Renner86}. 
The elements of $R_n$ are called {\em rooks}, or {\em rook matrices}. 
It is shown by Szechtman in \cite{Szechtman07} that each orbit closure in $B(Q;B_n)$ has a 
unique corresponding symmetric rook in $R_n$. Following \cite{BC12}, 
we call these rooks {\em partial involutions} as they satisfy the quadratic equation
$$
x^2 = e,
$$ 
where $e\in R_n$ is a diagonal matrix. We denote the set of all partial involutions in $R_n$ by $P_n$.

The {\em Bruhat-Chevalley-Renner ordering} on rooks is defined by 
$$
r \leq t \iff B_nr B_n \subseteq \overline{B_n t B_n},\ r,t\in R_n.
$$ 
Here, bar in our notation stands for the Zariski closure in $M$. 
Corresponding partial order on $P_n$, denoted by $\preceq$ is defined similarly; 
if $A$ and $A'$ are two $B_n$-orbit closures in $B(Q;B_n)$, and, $r$ and $r'$ 
are two partial involutions representing $A$ and $A'$, respectively, then 
$r \preceq r' \iff A'  \subseteq A$.

There is a simple combinatorial description for $\preceq$, which is due to Bagno and Cherniavsky, \cite{BC12}.
Let $X=(x_{ij})$ be an $n \times m$ matrix. For each $1 \leq k \leq n$ and $1 \leq l  \leq m$, denote by $X_{kl}$ 
the upper-left $k \times l$ submatrix of $X$. 
The {\em rank-control matrix} of $X$ is the $n \times m$ matrix $R(X) = (r_{kl})$ with entries given by 
$$
r_{kl}=rank(X_{kl}),
$$
for $1\leq k \leq n $ and $1\leq l \leq m$.
For example, the rank-control matrix of the partial involution 
$x= \begin{pmatrix}
1&0&0\\
0&0&0\\
0&0&1
\end{pmatrix}$ is 
\begin{align}\label{A:rank-control example 1}
R(x) =
\begin{pmatrix}
1&1&1\\
1&1&1\\
1&1&2
\end{pmatrix}.
\end{align}

For two matrices $A=(a_{kl})$ and $B=(b_{kl})$ of the same size with integer entries, we write $A\leq _R B$, if
$a_{kl}\leq b_{kl}$ for all $k$ and $l$. Then 
$$
x \preceq y\ \text{if and only if}\ R(x)\leq_R R(y).
$$
Although $\preceq$ is more natural from a geometric point of view, we prefer to work with its opposite, 
which we denote, by abuse of notation, by $\leq$, also. 
We depict the Hasse diagram of the opposite partial order on 
partial involutions for $n=3$ in Figure \ref{F:Involutions3} below.

Our first main result is that 
\begin{Theorem}\label{T:first main result}
The poset $(P_n, \leq)$ is lexicographically shellable.
\end{Theorem}

Let $S_n$ denote the symmetric group of permutations, which is contained in $R_n$ as the group of invertible rooks. 
In an increasing order of generality, the articles \cite{Edelman81}, \cite{Proctor82}, and \cite{BW82} show that 
$(S_n,\leq)$ is a lexicographically shellable poset. Generalizing this result to $R_n$, \cite{Can12} shows that 
$R_n$ is a lexicographically shellable poset. See also the related work \cite{Putcha02} of Putcha.

Recall that a graded poset $(P,\leq)$ with the rank function $\rho : P \rightarrow \N$ is called {\em Eulerian}, 
if for all $x \leq y$ the equality 
$$
| \{ z\in [x,y]:\ \rho(z)\ \text{is odd} \}| = | \{ z\in [x,y]:\ \rho(z)\ \text{is even} \}|
$$
holds.

Let $I_n \subset P_n$ denote the subset consisting of invertible involutions. 
It is shown by Incitti in \cite{Incitti04} that $I_n$ with its ``opposite inclusion ordering'' is not only 
lexicographically shellable but also Eulerian. Unfortunately, neither $R_n$ nor $P_n$ is Eulerian, 
so, we direct our attention to certain important subposets of them.

Let $H$ denote the group of invertible elements of a monoid $N$. 
It is important for semigroup theorists to understand the structure of orbits of $H$ on $N$ for various actions. 
In this regard, we consider the following action of $S_n \times S_n$ on $R_n$: 
\begin{align}\label{A:action rooks}
(x,y)\cdot z = x z y^{-1},\ \text{for all}\ z\in R_{n},\ x,y\in S_n.
\end{align}
There is natural restriction of this action to its diagonal subgroup 
$S_n \simeq \Delta S_n \hookrightarrow S_n\times S_n$ on partial involutions:
\begin{align}\label{A:action involutions}
y\cdot t =  (y^{-1})^\top t y^{-1},\ \text{for}\ t\in P_{n},\ y\in S_n.
\end{align}
Let $R_{n,k} \subset R_n$ denote the set of rook matrices with $k$ non-zero entries, 
and let $P_{n,k} = R_{n,k} \cap P_n$. 
Then any orbit of (\ref{A:action rooks}) is equal to one of $R_{n,k}$ for some $k$, and similarly, 
any orbit of (\ref{A:action involutions}) is equal to one of $P_{n,k}$ for some $k$.

Once $k$ is fixed, the unions $\cup_{l\leq k} R_{n,l}$ and $\cup_{l\leq k} P_{n,l}$ parametrize Borel orbits 
in certain determinental varieties. Therefore, it is important to study the restriction of Bruhat-Chevalley-Renner 
ordering on these subposets.

Although they are significantly different from each other, $R_{n,k}$ and $P_{n,k}$ share many important properties.
For example, both of them have the smallest and the largest elements.  
In fact, much more is true.

Given a positive integer $n$, let $[n]$ denote the set $\{1,2, \dots, n\}$. 
\begin{Theorem}\label{T:Eulerian}
For all $n\geq 1$ and $k\in [n]$, the subposets $R_{n,k}\subseteq R_n$ and $P_{n,k} \subseteq P_n$ are 
Eulerian if and only if $k=n$ or $k=n-1$. 
\end{Theorem}

The proof of Theorem \ref{T:Eulerian} relies on the following intriguing result:
\begin{Theorem}\label{T:union isomorphisms}
For all $n\geq 1$, 
\begin{enumerate}
\item $(R_{n,n-1} \cup R_{n,n},\leq )$ is isomorphic to the poset $(S_{n+1},\leq )$, and
\item $(P_{n,n-1} \cup P_{n,n}, \leq)$ is isomorphic to the poset $(I_{n+1},\leq)$. 
\end{enumerate}
\end{Theorem}

Let us point out that, as a remarkable corollary of Theorems \ref{T:first main result} and \ref{T:union isomorphisms} 
together with the main result of \cite{Can12}, we obtain new $EL$-labelings of $S_{n+1}$ and $I_{n+1}$ 
induced from their imbedding into $R_n$ and $P_n$, respectively.

We conclude our introduction by giving some references to the current developments. 
It is well known that the $B_n$-orbits under congruence action on invertible, $2n\times 2n$ skew-symmetric matrices 
are parametrized by fixed-point-free involutions of $S_{2n}$. 
In \cite{Cherniavsky11}, extending this fact to the space of all skew-symmetric matrices, Cherniavsky shows that the Borel orbits 
are parametrized by the partial fixed-point-free involutions in $R_{2n}$. 
Furthermore, in the same paper Cherniavsky gives a combinatorial description of the inclusion ordering of the 
Borel orbit closures in terms of rank control matrices of the partial fixed-point-free involutions. 
Recently, in \cite{CCT13} the authors together with Cherniavsky show that the inclusion poset of invertible fixed-point-free involutions is 
$EL$-shellable. 
More recently, extending the results of \cite{CCT13}, 
the second author shows that the inclusion poset of partial fixed-point-free involutions is $EL$-shellable. 
This result will appear in a forthcoming article.

The organization of our paper is as follows. In Section \ref{S:definitions} we introduce our notation
and provide preliminaries. In Section \ref{S:covering relations} we study covering relations of the opposite order 
$\leq$. 
In Section \ref{S:proof1} we prove Theorem \ref{T:first main result},
and finally, in Section \ref{S:proof2} we prove Theorems \ref{T:Eulerian} and \ref{T:union isomorphisms}.

\vspace{.5cm}

\noindent \textbf{Acknowledgement.}
The authors thank the anonymous referee for the constructive comments and careful reading of the article.
Both of the authors are partially supported by the Louisiana Board of Regents enhancement grant.

\section{\textbf{Background.}}\label{S:definitions}

\subsection{Lexicographic shellability.}

We start with reviewing the notion of lexicographic shellability. Our first remark is that 
in literature there are various versions of lexicographic shellability and the one we introduce
here is known to imply all of the others.

Let $P$ be a finite poset with a maximum and a minimum element, denoted by $\hat{1}$ and $\hat{0}$, respectively.  
We assume that $P$ is \textit{graded} of {\em rank} $n$. In other words, all maximal chains of $P$ have equal length $n$.
Denote by $C(P)$ the set of covering relations 
\begin{equation*}
C(P)= \{(x,y)\in P\times P:\ y\ \text{covers}\ x \}.
\end{equation*}
An \textit{edge-labeling} on $P$ is a map $f=f_{P,\varGamma}: C(P) \rightarrow \varGamma$ 
into some totally ordered set $\varGamma$.
The  \textit{Jordan-H\"{o}lder sequence} (with respect to $f$) of a maximal chain 
$\mathfrak{c}: x_0 < x_1< \cdots < x_{n-1}< x_n$ of $P$ 
is the $n$-tuple 
\begin{equation*}
f(\mathfrak{c}):= (f((x_0,x_1)), f((x_1,x_2)),\dots, f((x_{n-1},x_n))) \in \varGamma^n.
\end{equation*}
Fix an edge labeling $f$, and a maximal chain $\mathfrak{c}:\ x_0< x_1 < \cdots < x_n$. 
We call both the maximal chain $\mathfrak{c}$  and its image $f(\mathfrak{c})$ {\em increasing}, if 
$$
f((x_0,x_1)) \leq f((x_1,x_2)) \leq \cdots \leq f((x_{n-1},x_n))
$$ 
holds in $\varGamma$. 

Let $k>0$ be a positive integer and let $\varGamma^k$ denote the 
$k-$fold cartesian product $\varGamma^k = \varGamma \times \cdots \times \varGamma $,
totally ordered with respect to the lexicographic ordering. 
An edge labeling $f:C(P) \rightarrow \varGamma$ is called an {\em $EL-$labeling}, if
\begin{enumerate}
	
\item in every interval $[x,y] \subseteq P$ of rank $k>0$ there exists a unique maximal chain $\mathfrak{c}$ such that 
$f(\mathfrak{c}) \in \varGamma^k$ is increasing,
	
\item the Jordan-H\"{o}lder sequence $f(\mathfrak{c}) \in \varGamma^k$ of the unique chain 
$\mathfrak{c}$ from (1) is the smallest among the Jordan-H\"{o}lder sequences of maximal chains 
$x=x_0< x_1 < \cdots < x_k=y$.
\end{enumerate}
A poset $P$ is called {\em EL-shellable}, if it has an $EL-$labeling.

\subsection{Rooks and their enumeration.}

We set up our notation for rook matrices and establish a preliminary enumerative result.

Let $x=(x_{ij}) \in R_n$ be a rook matrix of size $n$. Define the sequence $(a_1,\dots,a_n)$ by
\begin{equation}\label{E:oneline}
a_j = 
\begin{cases}
0  &\text{if the $j$'th column consists of zeros,}\\
i &\text{if $x_{ij}=1$.}
\end{cases}
\end{equation}
By abuse of notation, we denote both the matrix and the sequence $(a_1,\dots,a_n)$ by $x$.
For example, the associated sequence of the partial permutation matrix 
\begin{equation*}
x=\begin{pmatrix}  
0 & 0 & 0 & 0 \\
0 & 0 & 0 & 0 \\
1 & 0 & 0 & 0 \\ 
0 & 0 & 1 & 0
\end{pmatrix}
\end{equation*}
is $x=(3,0,4,0)$. 

Once $n$ is fixed, a rook matrix $x\in R_n$ with $k$-nonzero entries is called a {\em $k$-rook}.
Observe that the number of $k$-rooks is given by the formula 
\begin{align}\label{A:number of k rooks}
|R_{n,k}| = k! \cdot {n \choose k}^2.
\end{align}
Indeed, to determine a $k$-rook, we first choose $n-k$ 0 zero rows and $n-k$ 0 columns. 
This is done in ${n \choose n-k}^2$ ways. Next we decide for the non-zero entires of the $k$-rook. 
Since deleting the zero rows and columns results in a permutation matrix of size $k$, there are $k!$
possibilities. 
Hence, the formula follows.

Let $\tau_n$ denote the number of invertible partial involutions. By default, we set $\tau_0 =1$. 

There is no closed formula for $\tau_n$, however, there is a simple recurrence that it satisfies;  
\begin{align}\label{A:involution recurrence}
\tau_{n+1} = \tau_{n} + (n-1) \tau_{n-1}\ (n\geq 1).
\end{align}
There is a similar recurrence satisfied by the number of invertible $n$-rooks (permutations);
\begin{align}\label{A:factorial recurrence}
(n+1)! = n ! + n^2 \cdot (n-1)!\ (n\geq 1).
\end{align}
It follows that 

\begin{Lemma}\label{L:supporting}
For all $n\geq 1$, 
\begin{enumerate}
\item $|R_{n,n-1} \cup R_{n,n}| = (n+1)!$,
\item $|P_{n,n-1} \cup P_{n,n}| = \tau_{n+1}$.
\end{enumerate}
\end{Lemma}
\begin{proof}
The first assertion follows from equations (\ref{A:factorial recurrence}) and (\ref{A:number of k rooks}).
The second assertion follows from equation (\ref{A:involution recurrence}) and the fact that 
$|P_{n,n-1}|= (n-1)\tau_{n-1}$.
\end{proof}

\subsection{An $EL$-labeling of invertible involutions.}\label{S:invertible involutions} 
	
In \cite{Incitti04}, Incitti shows that the poset of invertible involutions is $EL$-shellable. 
Let us briefly recall his arguments. 

For a permutation $\sigma \in S_n$, a {\em rise} of $\sigma$ is a pair $(i,j) \in [n]\times [n]$ such that 
$$
i<j\ \text{and}\ \sigma(i)<\sigma(j).
$$
A rise $(i,j)$ is called {\em free}, if there is no $k \in [n]$ such that 
$$
i<k<j\ \text{and}\ \sigma(i)<\sigma(k)<\sigma(j).
$$
For $\sigma \in S_n$, define its {\em fixed point set, its exceedance set} and its {\em defect set} to be 
\begin{align*}
I_f(\sigma) &=Fix(\sigma)=\{i \in [n]:\sigma(i)=i\}, \\
I_e(\sigma) &=Exc(\sigma)=\{i \in [n]:\sigma(i)>i\}, \\
I_d(\sigma) &=Def(\sigma)=\{i \in [n]:\sigma(i)<i\},
\end{align*}
respectively.

The {\em type} of a rise $(i,j)$ is defined to be the pair $(a,b)$, if 
$i \in I_a(\sigma)$ and $j \in I_b(\sigma)$, for some $a,b \in \{f,e,d\}$.
We call a rise of type $(a,b)$ an {\em $ab$-rise}.
Two kinds of $ee$-rises have to be distinguished from each other; 
an $ee$-rise is called {\em crossing}, if 
$i<\sigma(i)<j<\sigma(j)$, and it is called {\em non-crossing}, if $i<j<\sigma(i)<\sigma(j)$.
The rise $(i,j)$ of an involution $\sigma \in I_n$ is called {\em suitable}, 
if it is free and if its type is one of the following: $(f,f), (f,e), (e,f), (e,e), (e,d).$
We depict these possibilities in the first two columns in Figure \ref{CTofIncitti}, below. 
It is easy to check that the involution $\tau$ that is on the rightmost column in Figure \ref{CTofIncitti} 
covers $\sigma$. 
In this case, the covering relation is called a {\em covering transformation} of type $(i,j)$,
and $\tau$ is denoted by $ct_{(i,j)}(\sigma)$. 
In \cite{Incitti04}, Incitti shows that covering transformations exhaust all possible covering relations in $I_n$, 
and moreover, he shows that the labeling 
$$
F((\sigma, ct_{(i,j)}(\sigma))) := (i,j) \in [n]\times [n]
$$ 
is an $EL$-labeling for $I_n$. 
In the next section we extend Incitti's result to the set all partial involutions.

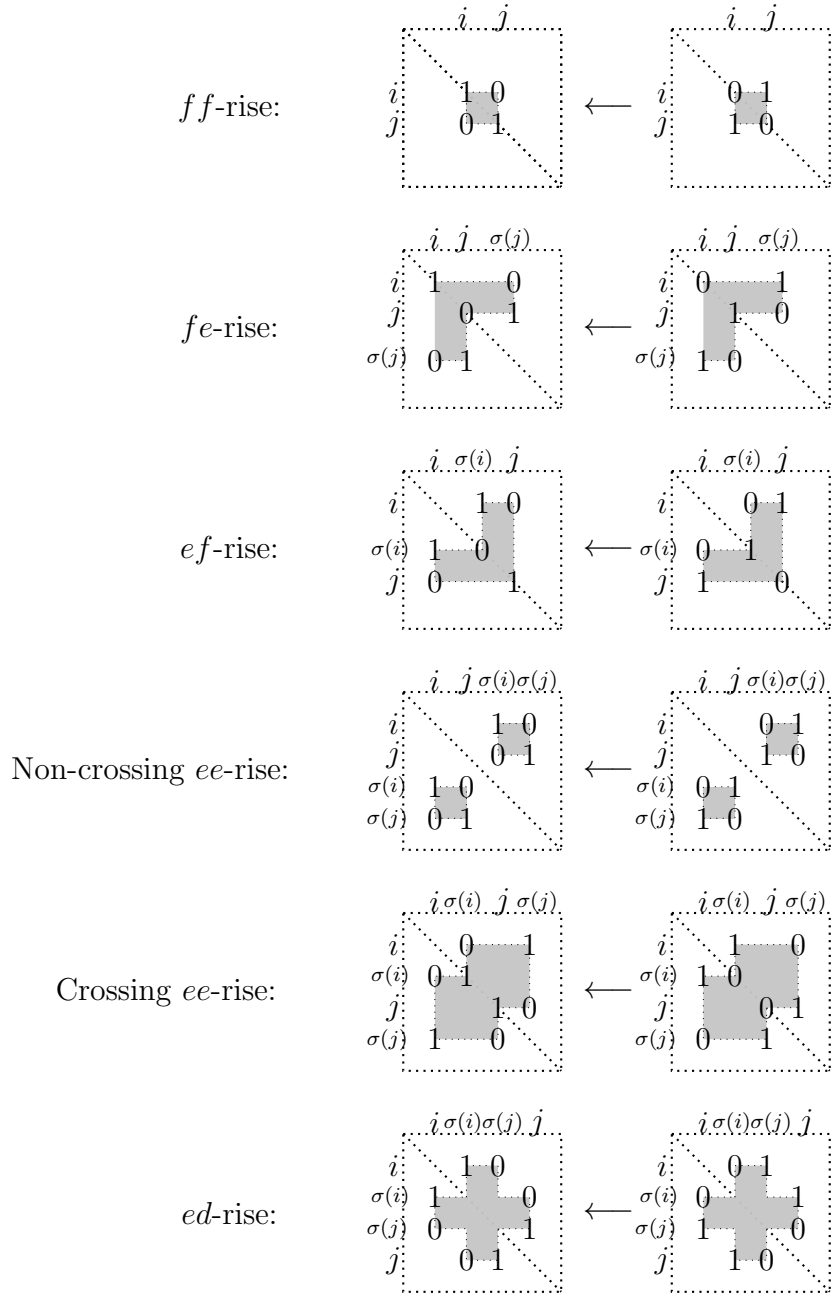
\begin{figure}[htp]
\begin{center}
\begin{tikzpicture}[scale=.42]

\begin{scope}[xshift=0,yshift=17.5cm]

\begin{scope}[xshift=-10cm,yshift=0cm]
\node at (0,0) {$ff$-rise:};
\end{scope}

\begin{scope}[xshift=2cm,yshift=0cm]
\node at (0,0) {$\longleftarrow$};
\end{scope}

\begin{scope}[xshift=-2cm,yshift=0cm]
\draw[dotted, thick,black] (-2.5,2.5) -- (2.5,2.5) -- (2.5,-2.5) -- (-2.5,-2.5) -- (-2.5,2.5);
\draw[dotted, thick, black] (-2.5,2.5) -- (2.5,-2.5);
\draw[dotted, thick,black] (-2.5,2.5) -- (2.5,2.5) -- (2.5,-2.5) -- (-2.5,-2.5) -- (-2.5,2.5);
\draw[dotted, thick, black] (-2.5,2.5) -- (2.5,-2.5);
\filldraw[dotted,fill=gray!45,opacity=.95]  (-.5,.5) -- (.5,.5) -- (.5,-.5) -- (-.5,-.5) -- (-.5,.5) ;
\node at (.5,.5) {0};
\node at (-.5,-.5) {0};
\node at (-.5,.5) {1};
\node at (.5,-.5) {1};
\node at (-2.8,.5) {$i$};
\node at (-2.8,-.5) {$j$};
\node at (-.6,2.85) {$i$};
\node at (.65,2.85) {$j$};
\end{scope}

\begin{scope}[xshift=6.5cm,yshift=0cm]
\draw[dotted, thick,black] (-2.5,2.5) -- (2.5,2.5) -- (2.5,-2.5) -- (-2.5,-2.5) -- (-2.5,2.5);
\draw[dotted, thick, black] (-2.5,2.5) -- (2.5,-2.5);
\filldraw[dotted,fill=gray!45,opacity=.95]  (-.5,.5) -- (.5,.5) -- (.5,-.5) -- (-.5,-.5) -- (-.5,.5) ;
\node at (.5,.5) {1};
\node at (-.5,-.5) {1};
\node at (-.5,.5) {0};
\node at (.5,-.5) {0};
\node at (-2.8,.5) {$i$};
\node at (-2.8,-.5) {$j$};
\node at (-.6,2.85) {$i$};
\node at (.65,2.85) {$j$};
\end{scope}

\end{scope}

\begin{scope}[xshift=0,yshift=10.5cm]

\begin{scope}[xshift=-10cm,yshift=0cm]
\node at (0,0) {$fe$-rise:};
\end{scope}

\begin{scope}[xshift=2cm,yshift=0cm]
\node at (0,0) {$\longleftarrow$};
\end{scope}

\begin{scope}[xshift=-2cm,yshift=0cm]
\draw[dotted, thick,black] (-2.5,2.5) -- (2.5,2.5) -- (2.5,-2.5) -- (-2.5,-2.5) -- (-2.5,2.5);
\draw[dotted, thick, black] (-2.5,2.5) -- (2.5,-2.5);
\filldraw[dotted,fill=gray!45,opacity=.95] (-1.5,1.5) -- (1,1.5) -- (1,.5) -- (-.5,.5) -- (-.5,-1) -- (-1.5,-1);
\node at (-1.5,1.5) {1};
\node at (1,.5) {1};
\node at (-.5,-1) {1};
\node at (-.5,.5) {0};
\node at (1,1.5) {0};
\node at (-1.5,-1) {0};
\node at (-2.75,1.5) {$i$};
\node at (-2.75,.5) {$j$};
\node at (-3,-.9) {\scriptsize{$\sigma(j)$}};
\node at (-1.5,2.85) {$i$};
\node at (-.6,2.85) {$j$};
\node at (.9,2.85) {\scriptsize{$\sigma(j)$}};
\end{scope}

\begin{scope}[xshift=6.5cm,yshift=0cm]
\draw[dotted, thick,black] (-2.5,2.5) -- (2.5,2.5) -- (2.5,-2.5) -- (-2.5,-2.5) -- (-2.5,2.5);
\draw[dotted, thick, black] (-2.5,2.5) -- (2.5,-2.5);
\filldraw[dotted,fill=gray!45,opacity=.95] (-1.5,1.5) -- (1,1.5) -- (1,.5) -- (-.5,.5) -- (-.5,-1) -- (-1.5,-1);
\node at (-1.5,1.5) {0};
\node at (1,.5) {0};
\node at (-.5,-1) {0};
\node at (-.5,.5) {1};
\node at (1,1.5) {1};
\node at (-1.5,-1) {1};
\node at (-2.75,1.5) {$i$};
\node at (-2.75,.5) {$j$};
\node at (-3,-.9) {\scriptsize{$\sigma(j)$}};
\node at (-1.5,2.85) {$i$};
\node at (-.6,2.85) {$j$};
\node at (.9,2.85) {\scriptsize{$\sigma(j)$}};
\end{scope}

\end{scope}

\begin{scope}[xshift=0,yshift=3.5cm]

\begin{scope}[xshift=-10cm,yshift=0cm]
\node at (0,0) {$ef$-rise:};
\end{scope}

\begin{scope}[xshift=2cm,yshift=0cm]
\node at (0,0) {$\longleftarrow$};
\end{scope}

\begin{scope}[xshift=-2cm,yshift=0cm]
\draw[dotted, thick,black] (-2.5,2.5) -- (2.5,2.5) -- (2.5,-2.5) -- (-2.5,-2.5) -- (-2.5,2.5);
\draw[dotted, thick, black] (-2.5,2.5) -- (2.5,-2.5);
\filldraw[dotted,fill=gray!45,opacity=.95] (0,0) -- (0,1.5) -- (1,1.5)-- (1,-1)  -- (-1.5,-1) 
-- (-1.5, 0) -- (0,0);
\node at (0,1.5) {1};
\node at (-1.5,0) {1};
\node at (1,-1) {1};
\node at (0,0) {0};
\node at (1,1.5) {0};
\node at (-1.5,-1) {0};
\node at (-2.8,1.5) {$i$};
\node at (-2.9,0) {\scriptsize{$\sigma(i)$}};
\node at (-2.8,-1) {$j$};
\node at (-1.5, 2.85) {$i$};
\node at (-.25, 2.85) {\scriptsize{$\sigma(i)$}};
\node at (1, 2.85) {$j$};
\end{scope}

\begin{scope}[xshift=6.5cm,yshift=0cm]
\draw[dotted, thick,black] (-2.5,2.5) -- (2.5,2.5) -- (2.5,-2.5) -- (-2.5,-2.5) -- (-2.5,2.5);
\draw[dotted, thick, black] (-2.5,2.5) -- (2.5,-2.5);
\filldraw[dotted,fill=gray!45,opacity=.95] (0,0) -- (0,1.5) -- (1,1.5)-- (1,-1)  -- (-1.5,-1) 
-- (-1.5, 0) -- (0,0);
\node at (0,1.5) {0};
\node at (-1.5,0) {0};
\node at (1,-1) {0};
\node at (0,0) {1};
\node at (1,1.5) {1};
\node at (-1.5,-1) {1};
\node at (-2.8,1.5) {$i$};
\node at (-2.9,0) {\scriptsize{$\sigma(i)$}};
\node at (-2.8,-1) {$j$};
\node at (-1.5, 2.85) {$i$};
\node at (-.25, 2.85) {\scriptsize{$\sigma(i)$}};
\node at (1, 2.85) {$j$};
\end{scope}

\end{scope}

\begin{scope}[xshift=0,yshift=-3.5cm]

\begin{scope}[xshift=-12.5cm,yshift=0cm]
\node at (0,0) {Non-crossing $ee$-rise:};
\end{scope}

\begin{scope}[xshift=2cm,yshift=0cm]
\node at (0,0) {$\longleftarrow$};
\end{scope}

\begin{scope}[xshift=-2cm,yshift=0cm]
\draw[dotted, thick,black] (-2.5,2.5) -- (2.5,2.5) -- (2.5,-2.5) -- (-2.5,-2.5) -- (-2.5,2.5);
\draw[dotted, thick, black] (-2.5,2.5) -- (2.5,-2.5);
\filldraw[dotted,fill=gray!45,opacity=.95] (.5,.5) -- (1.5,.5) -- (1.5,1.5) -- (.5,1.5) -- (.5,.5) ;
\filldraw[dotted,fill=gray!45,opacity=.95] (-.5,-.5) -- (-1.5,-.5) -- (-1.5,-1.5) -- (-.5,-1.5) -- (-.5,-.5) ;
\node at (.5,.5) {0};
\node at (1.5,.5) {1};
\node at (1.5,1.5) {0};
\node at (.5,1.5) {1};
\node at (-.5,-.5) {0};
\node at (-1.5,-.5) {1};
\node at (-1.5,-1.5) {0};
\node at (-.5,-1.5) {1};
\node at (-2.8,1.5) {$i$};
\node at (-2.8,.5) {$j$};
\node at (-3,-.5) {\scriptsize{$\sigma(i)$}};
\node at (-3,-1.5) {\scriptsize{$\sigma(j)$}};
\node at (-1.5,2.85) {$i$};
\node at (-.5,2.85) {$j$};
\node at (.5,2.85) {\scriptsize{$\sigma(i)$}};
\node at (1.75,2.85) {\scriptsize{$\sigma(j)$}};
\end{scope}

\begin{scope}[xshift=6.5cm,yshift=0cm]
\draw[dotted, thick,black] (-2.5,2.5) -- (2.5,2.5) -- (2.5,-2.5) -- (-2.5,-2.5) -- (-2.5,2.5);
\draw[dotted, thick, black] (-2.5,2.5) -- (2.5,-2.5);
\filldraw[dotted,fill=gray!45,opacity=.95] (.5,.5) -- (1.5,.5) -- (1.5,1.5) -- (.5,1.5) -- (.5,.5) ;
\filldraw[dotted,fill=gray!45,opacity=.95] (-.5,-.5) -- (-1.5,-.5) -- (-1.5,-1.5) -- (-.5,-1.5) -- (-.5,-.5) ;
\node at (.5,.5) {1};
\node at (1.5,.5) {0};
\node at (1.5,1.5) {1};
\node at (.5,1.5) {0};
\node at (-.5,-.5) {1};
\node at (-1.5,-.5) {0};
\node at (-1.5,-1.5) {1};
\node at (-.5,-1.5) {0};
\node at (-2.8,1.5) {$i$};
\node at (-2.8,.5) {$j$};
\node at (-3,-.5) {\scriptsize{$\sigma(i)$}};
\node at (-3,-1.5) {\scriptsize{$\sigma(j)$}};
\node at (-1.5,2.85) {$i$};
\node at (-.5,2.85) {$j$};
\node at (.5,2.85) {\scriptsize{$\sigma(i)$}};
\node at (1.75,2.85) {\scriptsize{$\sigma(j)$}};
\end{scope}

\end{scope}

\begin{scope}[xshift=0,yshift=-10.5cm]

\begin{scope}[xshift=-12cm,yshift=0cm]
\node at (0,0) {Crossing $ee$-rise:};
\end{scope}

\begin{scope}[xshift=2cm,yshift=0cm]
\node at (0,0) {$\longleftarrow$};
\end{scope}

\begin{scope}[xshift=-2cm,yshift=0cm]
\draw[dotted, thick,black] (-2.5,2.5) -- (2.5,2.5) -- (2.5,-2.5) -- (-2.5,-2.5) -- (-2.5,2.5);
\draw[dotted, thick, black] (-2.5,2.5) -- (2.5,-2.5);
\filldraw[dotted,fill=gray!45,opacity=.95] (-.5,.5) -- (-.5,1.5) -- (1.5,1.5)-- (1.5, -.5)  -- (.5,-.5) -- (.5,-1.5) -- (-1.5,-1.5)-- (-1.5,-.5) -- (-1.5, .5) -- (-.5,.5);
\node at (-.5,1.5) {0};
\node at (1.5,-.5) {0};
\node at (.5,-1.5) {0};
\node at (-1.5,.5) {0};
\node at (1.5,1.5) {1};
\node at (-1.5,-1.5) {1};
\node at (-.5,.5) {1};
\node at (.5,-.5) {1};
\node at (-2.8,1.5) {$i$};
\node at (-2.9,.5) {\scriptsize{$\sigma(i)$}};
\node at (-2.8,-.5) {$j$};
\node at (-3,-1.5) {\scriptsize{$\sigma(j)$}};
\node at (-1.5,2.85) {$i$};
\node at (-.6,2.85) {\scriptsize{$\sigma(i)$}};
\node at (.65,2.85) {$j$};
\node at (1.75,2.85) {\scriptsize{$\sigma(j)$}};
\end{scope}

\begin{scope}[xshift=6.5cm,yshift=0cm]
\draw[dotted, thick,black] (-2.5,2.5) -- (2.5,2.5) -- (2.5,-2.5) -- (-2.5,-2.5) -- (-2.5,2.5);
\draw[dotted, thick, black] (-2.5,2.5) -- (2.5,-2.5);
\filldraw[dotted,fill=gray!45,opacity=.95] (-.5,.5) -- (-.5,1.5) -- (1.5,1.5)-- (1.5, -.5)  -- (.5,-.5) -- (.5,-1.5) -- (-1.5,-1.5)-- (-1.5,-.5) -- (-1.5, .5) -- (-.5,.5);
\node at (-.5,1.5) {1};
\node at (1.5,-.5) {1};
\node at (.5,-1.5) {1};
\node at (-1.5,.5) {1};
\node at (1.5,1.5) {0};
\node at (-1.5,-1.5) {0};
\node at (-.5,.5) {0};
\node at (.5,-.5) {0};
\node at (-2.8,1.5) {$i$};
\node at (-2.9,.5) {\scriptsize{$\sigma(i)$}};
\node at (-2.8,-.5) {$j$};
\node at (-3,-1.5) {\scriptsize{$\sigma(j)$}};
\node at (-1.5,2.85) {$i$};
\node at (-.6,2.85) {\scriptsize{$\sigma(i)$}};
\node at (.65,2.85) {$j$};
\node at (1.75,2.85) {\scriptsize{$\sigma(j)$}};
\end{scope}

\end{scope}

\begin{scope}[xshift=0,yshift=-17.5cm]

\begin{scope}[xshift=-10cm,yshift=0cm]
\node at (0,0) {$ed$-rise:};
\end{scope}

\begin{scope}[xshift=2cm,yshift=0cm]
\node at (0,0) {$\longleftarrow$};
\end{scope}

\begin{scope}[xshift=-2cm,yshift=0cm]
\draw[dotted, thick,black] (-2.5,2.5) -- (2.5,2.5) -- (2.5,-2.5) -- (-2.5,-2.5) -- (-2.5,2.5);
\draw[dotted, thick, black] (-2.5,2.5) -- (2.5,-2.5);
\filldraw[dotted,fill=gray!45,opacity=.95] (-.5,.5) -- (-.5,1.5) -- (.5,1.5) -- (.5,.5) -- (1.5,.5) -- (1.5, -.5)  -- (.5,-.5) 
-- (.5,-1.5) -- (-.5,-1.5) -- (-.5, -.5) -- (-1.5,-.5) -- (-1.5, .5) -- (-.5,.5);
\node at (-.5,1.5) {1};
\node at (.5,1.5) {0};
\node at (1.5,.5) {0};
\node at (1.5,-.5) {1};
\node at (.5,-1.5) {1};
\node at (-.5,-1.5) {0};
\node at (-1.5,-.5) {0};
\node at (-1.5,.5) {1};
\node at (-2.8,1.5) {$i$};
\node at (-2.9,.5) {\scriptsize{$\sigma(i)$}};
\node at (-3,-.5) {\scriptsize{$\sigma(j)$}};
\node at (-2.8,-1.5) {$j$};
\node at (-1.5,2.85) {$i$};
\node at (-.6,2.85) {\scriptsize{$\sigma(i)$}};
\node at (.65,2.85) {\scriptsize{$\sigma(j)$}};
\node at (1.75,2.85) {$j$};
\end{scope}

\begin{scope}[xshift=6.5cm,yshift=0cm]
\draw[dotted, thick,black] (-2.5,2.5) -- (2.5,2.5) -- (2.5,-2.5) -- (-2.5,-2.5) -- (-2.5,2.5);
\draw[dotted, thick, black] (-2.5,2.5) -- (2.5,-2.5);
\filldraw[dotted,fill=gray!45,opacity=.95] (-.5,.5) -- (-.5,1.5) -- (.5,1.5) -- (.5,.5) -- (1.5,.5) -- (1.5, -.5)  -- (.5,-.5) 
-- (.5,-1.5) -- (-.5,-1.5) -- (-.5, -.5) -- (-1.5,-.5) -- (-1.5, .5) -- (-.5,.5);
\node at (-.5,1.5) {0};
\node at (.5,1.5) {1};
\node at (1.5,.5) {1};
\node at (1.5,-.5) {0};
\node at (.5,-1.5) {0};
\node at (-.5,-1.5) {1};
\node at (-1.5,-.5) {1};
\node at (-1.5,.5) {0};
\node at (-2.8,1.5) {$i$};
\node at (-2.9,.5) {\scriptsize{$\sigma(i)$}};
\node at (-3,-.5) {\scriptsize{$\sigma(j)$}};
\node at (-2.8,-1.5) {$j$};
\node at (-1.5,2.85) {$i$};
\node at (-.6,2.85) {\scriptsize{$\sigma(i)$}};
\node at (.65,2.85) {\scriptsize{$\sigma(j)$}};
\node at (1.75,2.85) {$j$};
\end{scope}

\end{scope}

\end{tikzpicture}
\end{center}
\caption{Covering transformations $\sigma \leftarrow \tau=ct_{(i,j)}(\sigma)$ of $I_n$.}
\label{CTofIncitti}
\end{figure}

\section{\textbf{Covering relations of partial involutions $P_n$.}}\label{S:covering relations}

Covering relations in $P_n$ depend on a numerical invariant associated with the rank control matrices.
For any non-negative integer $k$, define $r_{0,k}$ to be 0. For a rank-control matrix $R(X)=(r_{ij})$, define 
$$
D(x)=\#\{(i,j)|1 \leq i \leq j \leq n \ \text{and} \ r_{ij}=r_{i-1,j-1}\}.
$$	
For example, if $R(x)$ is as in (\ref{A:rank-control example 1}), then $D(x)= \# \{(2,2),(2,3)\} =2$.

In \cite{BC12}, Bagno and Cherniavsky prove that, in $(P_n,\preceq)$, 
\begin{align*}
x \ \text{covers} \ y \iff R(y) \leq_R R(x) \ \text{and} \ D(x)=D(y)+1.
\end{align*}
However, we need a finer classification of the covering types.  
The notion of suitable rise on involutions extends to the partial permutations $(P_n,\leq)$, verbatim. 
Of course, there are additional covering relations. In this section we exhibit all of them.

\begin{Lemma}\label{L:explicit covering relations for P_n}
Let $x$ and $y$ be two partial involutions. Then $x$ covers $y$ if and  
only if one of the following is true:

\begin{enumerate}

\item $x$ and $y$ have the same zero-rows and columns.
Let $\wil{x}$ and $\wil{y}$ denote the full rank involutions obtained from $x$ and $y$, respectively, 
by deleting common zero rows and columns. 
Then $x$ covers $y$ if and only if $\wil{x}$ covers $\wil{y}$.
For example,
$$
y=\begin{pmatrix}
1&0&0\\
0&0&0\\
0&0&1
\end{pmatrix}
\
\text{is covered by}
\
x=\begin{pmatrix}
0&0&1\\
0&0&0\\
1&0&0
\end{pmatrix}.
$$

\item Without removing a suitable rise, $x$ is obtained from $y$ by one of the following moves:

\begin{enumerate}

\item a $1$ on the diagonal is moved down to the first available  
diagonal entry.
It is possible for a $1$ to be pushed out of the matrix.  For example,
$$
y=\begin{pmatrix}
0&0&1&0\\
0&1&0&0\\
1&0&0&0\\
0&0&0&0\\
\end{pmatrix}
\
\text{is covered by}
\
x=\begin{pmatrix}
0&0&1&0\\
0&0&0&0\\
1&0&0&0\\
0&0&0&1\\
\end{pmatrix}.
$$

\item Two off-diagonal symmetric $1$'s are pushed right/down or down/ 
right to the first available entries at symmetric positions. There are two cases which we  
demonstrate by examples: 
\begin{enumerate}
\item
$\ds{
y=\begin{pmatrix}
0&1&0\\
1&0&0\\
0&0&0
\end{pmatrix}
\
\text{is covered by}
\
x=\begin{pmatrix}
0&0&1\\
0&0&0\\
1&0&0
\end{pmatrix},
}$

\item $\ds{
y=\begin{pmatrix}
0&0&1\\
0&0&0\\
1&0&0
\end{pmatrix}
\
\text{is covered by}
\
x=\begin{pmatrix}
0&0&0\\
0&0&1\\
0&1&0
\end{pmatrix}.}
$
\end{enumerate}

As a special case of $ii.$, if there are no available entries at symmetric positions to push
down and right, then the two 1's at positions $(i,j)$ and $(j,i)$ with $i> j$ are pushed
to $(i,i)$, and to the first available diagonal entry below $(i,i)$. 
For example,  
$$
y=\begin{pmatrix}
0& 0&1&0\\
0& 1&0&0\\
1& 0&0&0 \\
0&0&0&0
\end{pmatrix}
\
\text{is covered by}
\
x=\begin{pmatrix}
0&0&0 &0\\
0&1&0 &0\\
0&0&1& 0\\
0& 0 & 0& 1
\end{pmatrix}.
$$
In this case, a single 1 is allowed to be pushed out of the matrix. 
For example, 
$\ds{
y=\begin{pmatrix}
0&1\\
1&0
\end{pmatrix}
\
\text{is covered by}
\
x=\begin{pmatrix}
0&0\\
0&1
\end{pmatrix}.}$
\end{enumerate}
\end {enumerate}

\end{Lemma}

Before we start the proof, let us illustrate by an example, what it means to remove a 
suitable rise: 
\begin{Example}
Let 
$\ds{
y=\begin{pmatrix}
0&1 &0&0\\
1&0&0&0\\
0&0& 1& 0\\
0&0&0&0
\end{pmatrix}}
$
and let 
$\ds{
x=\begin{pmatrix}
0&0 &0&1\\
0&0&0&0\\
0&0& 1& 0\\
1&0&0&0
\end{pmatrix}}.
$
Then $x$ is obtained from $y$ by a move as in 2.(b)$i$., however,
it removes the suitable rise $(1,3)$. Therefore, it is not a covering relation. 
\end{Example}

\begin{proof} 
Comparing the rank-control matrices $R(\cdot)$ as well as the invariants $D(\cdot)$ of $x$ and $y$,  
the ``if'' direction of the claim is straightforward to verify.

We prove the ``only if'' direction by contraposition. To this end let $x$ denote a partial involution that covers $y\in P_n$,
and $x$ is not obtained by one the moves as in 1., 2.(a), or 2.(b).

Since 1. does not hold, $x$ has a smallest row consisting of zeros such that the corresponding row of 
$y$ contains a non-zero entry. Let $i$ denote the index of this zero row of $x$. 
Notice that, if there is a zero row for both $x$ and $y$ with the same index, then removing this row and
the corresponding column does not have any effect on the remaining entries of the rank-control matrices. 
Therefore, we assume that neither $x$ nor $y$ has a zero row before the $i$-th row.

There are two subcases; \\
I) the nonzero entry of $y$ occurs before the $i$-th column, \\
II) it occurs on the $i$-th column. 

Realize that occurring after the $i$-th column can also be interpreted as before the $i$-th column if we relabel the columns.

We proceed with I). Then $y$ and $x$ are as in 
\begin{align*}
y=
\begin{pmatrix}
&&& & 0 & & \\
&A&& & \vdots &B& \\
&&& & 1 & & & \\
&& & & \vdots &&& \\
0  & \cdots & 1 & \cdots & 0& \cdots &0 \\
&C&& & \vdots & D & \\
&& & & 0& &  
\end{pmatrix}\
\text{and}\ 
x= 
\begin{pmatrix}
&&& & 0 & & \\
&A'&& & \vdots &B'& \\
&&& & 0 & & & \\
&& & & \vdots &&& \\
0  & \cdots & 0 & \cdots & 0& \cdots &0 \\
&C'&& & \vdots & D' & \\
&& & & 0& &  
\end{pmatrix},
\end{align*} 
where $A,A'B,B',\dots$ stand for appropriate size matrices. 
Let $1\leq k < i$ denote the index of the row of $y$ with a 1 on its $i$-th entry.

Let $\varGamma$ denote the set of coordinates of non-zero entries $(r,s)$ of $x$ satisfying $k\leq r \leq n$ 
and $i < s \leq n$. Since the upper $k\times n$ portions of both of $y$ and $x$ are of rank $k$, 
$\varGamma \neq \emptyset$.

Let $(r,s) \in \varGamma$ denote the entry with smallest column index. 
Unless $r= s$, we define $x_1$ to be the matrix obtained from $x$ by moving the non-zero entries at the positions 
$(r,s)$ and $(s,r)$ (which exists, by symmetry) to the positions $(r,i)$ and $(i,r)$. 
If $r=s$, then $x_1$ is defined by moving the non-zero entry to the $(i,i)$-th position.

We claim that $y \leq x_1 < x$. Indeed, since $x_1$ is obtained from $x$ by reverse of the one of the moves 2.(a) or 2.(b), 
the second inequality is clear. The first inequality follows immediately from checking the corresponding rank-control matrices of 
$x$, $x_1$ and of $y$. Let us illustrate the procedure by two possible scenarios: 

\begin{Example}
Let 
\begin{align*}
y=
\begin{pmatrix}
0&0&0&0&1&0&0&0 \\ 
0&1&0&0&0&0&0&0 \\
0&0&1&0&0&0&0&0 \\
0&0&0&0&0&0&0&1 \\
1&0&0&0&0&0&0&0\\
0&0&0&0&0&1&0&0\\
0&0&0&0&0&0&0&0\\ 
0&0&0&1&0&0&0&0  
\end{pmatrix}\
\text{and}\ 
x= 
\begin{pmatrix}
0&0&0&0&0&1&0&0 \\ 
0&0&1&0&0&0&0&0 \\
0&1&0&0&0&0&0&0 \\
0&0&0&0&0&0&0&1 \\
0&0&0&0&0&0&0&0\\
1&0&0&0&0&0&0&0\\
0&0&0&0&0&0&1&0\\ 
0&0&0&1&0&0&0&0  
\end{pmatrix}.
\end{align*} 
Then $i=5$, $k=1$, and the rank-control matrices of $y$ and $x$ are 
\begin{align*}
R(y)=
\begin{pmatrix}
0&0&0&0&1&1&1&1 \\ 
0&1&1&1&2&2&2&2 \\
0&1&2&2&3&3&3&3 \\
0&1&2&2&3&3&3&4 \\
1&2&3&3&4&4&4&5\\
1&2&3&3&4&5&5&6\\
1&2&3&3&4&5&5&6\\
1&2&3&4&5&6&6&7\\
\end{pmatrix}\
\text{and}\ 
R(x)= 
\begin{pmatrix}
0&0&0&0&0&1&1&1 \\ 
0&0&1&1&1&2&2&2 \\
0&1&2&2&2&3&3&3 \\
0&1&2&2&2&3&3&4 \\
0&1&2&2&2&3&3&4 \\
1&2&3&3&3&4&4&5\\
1&2&3&3&3&4&5&6\\
1&2&3&4&4&5&6&7  
\end{pmatrix}.
\end{align*} 
In this case, 
\begin{align*}
x_1= 
\begin{pmatrix}
0&0&0&0&0&1&0&0 \\ 
0&0&1&0&0&0&0&0 \\
0&1&0&0&0&0&0&0 \\
0&0&0&0&0&0&0&1 \\
0&0&0&0&1&0&0&0\\
1&0&0&0&0&0&0&0\\
0&0&0&0&0&0&0&0\\ 
0&0&0&1&0&0&0&0  
\end{pmatrix}\
\text{and}\
R(x_1)=
\begin{pmatrix}
0&0&0&0&0&1&1&1 \\ 
0&0&1&1&1&2&2&2 \\
0&1&2&2&2&3&3&3 \\
0&1&2&2&2&3&3&4 \\
0&1&2&2&3&4&4&5 \\
1&2&3&3&4&5&5&6\\
1&2&3&3&4&5&5&6\\
1&2&3&4&5&6&6&7  
\end{pmatrix}.
\end{align*}

If 
\begin{align*}
y=
\begin{pmatrix}
0&0&0&0&1&0&0&0 \\ 
0&1&0&0&0&0&0&0 \\
0&0&1&0&0&0&0&0 \\
0&0&0&0&0&1&0&0 \\
1&0&0&0&0&0&0&0\\
0&0&0&1&0&0&0&0\\
0&0&0&0&0&0&1&0\\ 
0&0&0&0&0&0&0&1  
\end{pmatrix}\
\text{and}\ 
x= 
\begin{pmatrix}
0&0&0&0&0&1&0&0 \\ 
0&0&1&0&0&0&0&0 \\
0&1&0&0&0&0&0&0 \\
0&0&0&0&0&0&0&0 \\
0&0&0&0&0&0&1&0\\
1&0&0&0&0&0&0&0\\
0&0&0&0&1&0&0&0\\ 
0&0&0&0&0&0&0&0  
\end{pmatrix},
\end{align*} 
then $i=4$, $k=1$ and the rank-control matrices are 
\begin{align*}
R(y)=
\begin{pmatrix}
0&0&0&0&1&1&1&1 \\ 
0&1&1&1&2&2&2&2 \\
0&1&2&2&3&3&3&3 \\
0&1&2&2&3&4&4&4 \\
1&2&3&4&4&5&5&5\\
1&2&3&4&5&6&6&6\\
1&2&3&4&5&6&7&7\\
1&2&3&4&5&6&7&8\\
\end{pmatrix}\
\text{and}\ 
R(x)= 
\begin{pmatrix}
0&0&0&0&0&1&1&1 \\ 
0&0&1&1&1&2&2&2 \\
0&1&2&2&2&3&3&3 \\
0&1&2&2&2&3&3&3 \\
0&1&2&2&2&3&4&4 \\
1&2&3&3&3&4&5&5\\
1&2&3&3&4&5&6&6\\
1&2&3&3&4&5&6&6
\end{pmatrix}.
\end{align*} 
In this case, 
\begin{align*}
x_1= 
\begin{pmatrix}
0&0&0&0&0&1&0&0 \\ 
0&0&1&0&0&0&0&0 \\
0&1&0&0&0&0&0&0 \\
0&0&0&0&0&0&1&0 \\
0&0&0&0&0&0&0&0\\
1&0&0&0&0&0&0&0\\
0&0&0&1&0&0&0&0\\ 
0&0&0&0&0&0&0&0  
\end{pmatrix}\
\text{and}\
R(x_1)=
\begin{pmatrix}
0&0&0&0&0&1&1&1 \\ 
0&0&1&1&1&2&2&2 \\
0&1&2&2&2&3&3&3 \\
0&1&2&2&2&3&4&4 \\
0&1&2&2&2&3&4&4 \\
1&2&3&3&3&4&5&5\\
1&2&3&4&4&5&6&6\\
1&2&3&4&4&5&6&6\\
\end{pmatrix}.
\end{align*} 
\end{Example}

We proceed with case II) that the non-zero entry of $y$ in its $i$-th column occurs at the $k$-th row,
where $k \geq i$.

First of all, without loss of generality, we may assume that $x$ has a non-zero entry in its $i+1$-st row, whose column index
we denote by $j_x$. 

Denote by $y_1$ the partial involution obtained from $y$ by 
interchanging its $i$-th and $i+1$-st rows as well as interchanging its $i$-th and $i+1$-st columns. 
If it exists, let $j_y$ denote the column index of the non-zero entry of $y_1$ in its $i$-th row. If $j_y<k$, then, $y<y_1$. 
Furthermore, in this case, because $i$-th row of $x$ consists of 0's, $y_1 < x$. In other words, we have $y < y_1 < x$.

Therefore, we assume that $k<j_y$. In this case, if $k< j_x $, then let $x_1$ denote 
the partial involution obtained from $x$ by interchanging its $i$-th and $i+1$-st rows as well as interchanging its 
$i$-th and $i+1$-st columns. Then we have $y < x_1 < x$ and we are done. Therefore, we assume that $k > j_x$. 
But in this case $y < y_1 <x$ holds. 
This finishes the proof of the case 2), and we conclude the result.

\end{proof}

\section{\textbf{$EL$-labeling of $P_n$.}}\label{S:proof1}

In this section we define an edge labeling of $P_n$ and prove that it is an $EL$-labeling.

\begin{enumerate}

\item If the covering relation is derived from a regular covering of an involution,
namely from a move that is as in Lemma \ref{L:explicit covering relations for P_n}, Part 1., 
then we use the labeling as defined in \cite{Incitti04}.

\item If the covering relation results from a move as in Lemma \ref{L:explicit covering relations for P_n} Part 2.(a), 
namely from a diagonal push where the element that is pushed from is at the position $(i,i)$, 
then we label it by $(i,i)$. 

\item Suppose that a covering relation is as in Lemma \ref{L:explicit covering relations for P_n} (b). 
Observe that, in all of these covering relations, one of the 1's is pushed down and the other is pushed right. 
Let $i$ denote the column index of the first 1 that is pushed to the right, and let $j$ denote the index of 
the resulting column. Then we label the move by $(i,j)$.

\end{enumerate}

To illustrate the third labeling let us present a few examples. Also, see Figure \ref{F:Labeling of P3} below 
for the labeling of $P_3$ which is depicted in one-line notation.

\begin{Example}
$$
y=\begin{pmatrix}
0&0&0&1&0\\
0&0&1&0&0\\
0&1&0&0&0\\
1&0&0&0&0\\
0&0&0&0&0
\end{pmatrix}
\
\text{is covered by}
\
x=\begin{pmatrix}
0&0&0&1&0\\
0&0&0&0&1\\
0&0&0&0&0\\
1&0&0&0&0\\
0&1&0&0&0\end{pmatrix}
$$
The corresponding labeling here is $(3,5)$.

\end{Example}
\begin{Example}
$$
y=\begin{pmatrix}
0&0&0&0&0&1\\
0&0&0&0&1&0\\
0&0&0&0&0&0\\
0&0&0&0&0&0\\
0&1&0&0&0&0\\
1&0&0&0&0&0
\end{pmatrix}
\
\text{is covered by}
\
x=\begin{pmatrix}
0&0&0&0&0&0\\
0&0&0&0&1&0\\
0&0&0&0&0&1\\
0&0&0&0&0&0\\
0&1&0&0&0&0\\
0&0&1&0&0&0
\end{pmatrix}
$$
The corresponding labeling here is $(1,3)$.

\end{Example}
\begin{Example}
$$
y=\begin{pmatrix}
0&0&0&1&0\\
0&0&1&0&0\\
0&1&0&0&0\\
1&0&0&0&0\\
0&0&0&0&0
\end{pmatrix}
\
\text{is covered by}
\
x=\begin{pmatrix}
0&0&0&1&0\\
0&0&0&0&0\\
0&0&1&0&0\\
1&0&0&0&0\\
0&0&0&0&1
\end{pmatrix}
$$
The corresponding labeling here is $(2,3)$.

\end{Example}

If $x$ covers $y$ with label $(i,j)$, then we refer to it as an $(i,j)-covering$ and say that $y$ is obtained from $x$ by an $(i,j)$-move.
Alternatively, we call a covering relation a {\em $c$-cover}, if it is derived from an involution; 
a {\em $d$-cover}, if it is obtained by a shift of a diagonal element; an {\em $r$-cover}, if it is derived from a right/down or
a down/right move. We will refer to the corresponding moves as $c$-, $d$- and $r$-moves.

Let $\varGamma$ denote the lexicographic total order on the product $[n]\times [n]$. Then, for any $k>0,$ 
$\varGamma^k=\varGamma \times \cdots \times \varGamma$ is totally ordered with respect to lexicographic ordering. 
Finally, let $F:C(P_n) \rightarrow \varGamma$ denote the labeling function defined above.

For an interval $[x,y]\subseteq P_n$ and a maximal chain $\mathfrak{c}: x=x_0< \cdots < x_k=y$, 
we denote by $F(\mathfrak{c})$ the Jordan-H\"{o}lder sequence of labels of $\mathfrak{c}$:  
\begin{equation*}
F(\mathfrak{c})= (F((x_0,x_1)), \dots, F((x_{k-1},x_k))) \in \varGamma^k.
\end{equation*}

	\begin{Proposition}\label{P:lexicographically}
	Let $\mathfrak{c}: x=x_0< \cdots < x_k=y$ be a maximal chain in $[x,y]$ such that its 
	Jordan-H\"{o}lder sequence $F(\mathfrak{c})$ is lexicographically smallest among all 
	Jordan-H\"{o}lder sequences (of chains in $[x,y])$ in $\varGamma^k$. 
	 Then,
	\begin{equation}\label{E:lexfirst}
	F((x_0,x_1)) \leq F((x_1,x_2)) \leq \cdots \leq F((x_{k-1},x_k)).
	\end{equation}
	\end{Proposition}

\begin{proof}
Assume that (\ref{E:lexfirst}) is not true. Then, there exist three  
consecutive terms
\begin{equation*}
x_{t-1} < x_t < x_{t+1}
\end{equation*}
in $\mathfrak{c}$, such that $F((x_{t-1},x_t)) > F((x_t, x_{t+1}))$.
We have 9 cases to consider.

        \begin{center}
        
        Case 1: $type(x_{t-1},x_t)=c$, and $type(x_t,x_{t+1})=c$.\\

        Case 2: $type(x_{t-1},x_t)=d$, and $type(x_t,x_{t+1})=d$.\\

        Case 3: $type(x_{t-1},x_t)=d$, and $type(x_t,x_{t+1})=c$.\\

        Case 4: $type(x_{t-1},x_t)=c$, and $type(x_t,x_{t+1})=d$.\\

        Case 5: $type(x_{t-1},x_t)=r$, and $type(x_t,x_{t+1})=r$.\\

        Case 6: $type(x_{t-1},x_t)=d$, and $type(x_t,x_{t+1})=r$.\\

        Case 7: $type(x_{t-1},x_t)=r$, and $type(x_t,x_{t+1})=d$.\\

        Case 8: $type(x_{t-1},x_t)=r$, and $type(x_t,x_{t+1})=c$.\\

        Case 9: $type(x_{t-1},x_t)=c$, and $type(x_t,x_{t+1})=r$.

        \end{center}

In each of these 9 cases, we either produce an immediate contradiction  
by showing that we can interchange the two moves,
or we construct an element $z\in [x,y]$ which covers $x_{t-1}$, and  
such that
$F((x_{t-1},z))< F((x_{t-1},x_t))$. Since we assume that  
$F(\mathfrak{c})$ is the lexicographically first
Jordan-H\"{o}lder sequence, the existence of $z$ is a contradiction,  
too.

{\em Case 1:} Straightforward from the fact that type $c$ covering  
relations have identical labelings with Incitti's \cite{Incitti04}.

{\em Case 2:} Suppose that the first move is labeled $(i,i)$ and the  
second one $(j,j)$ with $j<i$.
If the two moves are not interchangeable then $(j,i)$ is a legal $c$-move in $x_{t-1}$. Since $(j,i)$ is lexicographically smaller  
than $(i,i)$, we derive a contradiction.

{\em Case 3:} Let $(i,i)$ be moved to $(j,j)$ in the first step (type  
$d$ move), hence $i< j$.
If the following $c$-move does not involve the entry at $(j,j)$, then  
either the $c$- and the $d$-move commute with each other,
or the rise for the $c$-move is not free in $x_{t-1}$.
In that case there has to be an $ef$-rise involving the entry at the  
position $(i,i)$.
This $ef$-rise has a smaller label than $(i,i)$, which is a  
contradiction.

Thus we may assume that the $c$-move involves the entry at the $(j,j)$-th position.
Then the $c$-move has to come from either an $ff$-, an $fe$-, or an $ef 
$-rise.

$ff$ is not possible:
Let $(a,b)$ the corresponding label.
The move involves the entry at the $(j,j)$-th position if either $a=j$ or $b=j$.
If $a=j$ then $(a,b)>(i,i)$ and the labels are increasing.
If $b=j$, then we must have $a<i$ for $(a,b)<(i,i)$. Therefore, there is a legal $c$-move $ 
(a,i)$ in $x_{t-1}$ has a smaller label than $(i,i)$.

$fe$ is not possible since $(j,b)$ is greater than $(i,i)$.

Finally, $ef$ is not possible: Let $(k,j)$ be the label of the $c$-move.
If $(k,i)$ is a suitable rise in $x_{t-1}$, then $(k,i)<(i,i)$.
If $(k,i)$ is not a suitable rise in $x_{t-1}$, let $(j,j), (k,l),$  
and $(l,k)$ denote the entries involved in the $c$-move where $l<k$.
Then $l<i<k$ and $(l,j)<(i,i)$. $(l,j)$ is a legal $r$-move in  
$x_{t-1}$ with a smaller label than $(i,i)$. This concludes case 3.

{\em Case 4:} This is not possible since no $c$-move places a $1$ on  
the diagonal such that moving this $1$
gives rise to a smaller labeling than the $c$-move.
Note that if there is a 1 on the diagonal before the $c$-move takes  
place, then moving this 1 first creates an element $z$ with covering label  
lexicographically smaller that of the $c$-move.
Thus we are done with this case.

{\em Case 5:}
Let the first move be labeled $(i,j)$ and the second $(k,l)$.
Then it is clear that if $k=j$ we obtain an increasing sequence.
If $k=i$, then we can switch the order of the moves.
If $k \neq \{i,j\}$, then, if possible, we  
perform the second move first.
If it is not possible to interchange the order of the $r$-moves, then  
by the first $r$-move
a suitable rise is removed. But then the corresponding $c$-cover has a  
smaller label in $x_{t-1}$ the $r$-move.

{\em Case 6:} We either perform the $r$-move first if possible, or perform the $c$-cover corresponding to the 
suitable rise removed by $d$-move which has a smaller label than the $d$-move in in $x_{t-1}$.

{\em Case 7:} Similar to \emph{Case 6} so we omit the proof.

{\em Case 8:} The $c$-move has to include the elements moved by the previous $r$-move since otherwise the 
$c$-move can be performed first.

If the suitable rise is created by the $r$-move then the label of the $r$-move is smaller than the label of the $c$-move.
Otherwise, there is a suitable rise in $x_{t-1}$ involving the elements moved by the $r$-move.
But the $c$-move corresponding to this suitable rise has a smaller label than the $r$-move.

{\em Case 9:} If the $r$-move does not involve an element moved by the $c$-move then perform the $r$-move first.
If this is not possible then a suitable rise is removed by moving it. The $c$-move corresponding to this suitable
rise has a smaller label than the other $c$-move.

If the $r$-move involves an element that is placed at this position by the preceding $c$-move, then we proceed to exhibit every $c$-move to exclude all of them:

$ff$: The label of $c$-move is $(i,j)$. The smaller $r$-move involving a new element can only be $(i,k)$ with $k<j$.
But then $(i,i)$ is possible in $x_{t-1}$ and $(i,i)<(i,j)$.

$fe$: Similar to \emph{ff} so we omit the proof.

$ef$: The label of $c$-move is $(i,j)$. The smaller $r$-move involving a new element can only be $(i,k)$ with $k<j$.
Then $(i,k)$ is possible in $x_{t-1}$ and $(i,k)<(i,j)$.

$\text{non-crossing}\ ee$, $\text{crossing}\ ee$ and $ed$ are similar to $ef$ so we omit the proof.

\end{proof}

	\begin{Proposition}\label{P:uniqueness} 
	We use the notation of Proposition \ref{P:lexicographically}. 
	There exists a unique maximal chain $x=x_0 < \cdots < x_k= y$ with $F((x_0,x_1)) \leq \cdots \leq F((x_{k-1}, x_k))$.
	\end{Proposition}

\begin{proof} 
We already know that the lexicographically first chain is increasing. 
Therefore, it is enough to show that there is no other increasing chain. 
We prove this by induction on the length of the interval $[x,y]$. 
Clearly, if $y$ covers $x$, there is nothing to prove. 
So, we assume that for any interval of length $k$ there exists a unique increasing maximal chain.

Let $[x,y] \subseteq P_n$ be an interval of length $k+1$, and let
\begin{equation*}
\mathfrak{c}: x=x_0< x_1 < \cdots < x_k <  x_{k+1}=y
\end{equation*}
be the maximal chain such that $F(\mathfrak{c})$ is the lexicographically first Jordan-H\"{o}lder sequence in $\varGamma^{k+1}$.

Assume that there exists another increasing chain
        \begin{equation*}
        \mathfrak{c}': x=x_0< x_1' < \cdots <x_k' <  x_{k+1}=y.
        \end{equation*}
Since the length of the chain
        \begin{equation*}
        x_1'< \cdots < x_k' < x_{k+1}=y
        \end{equation*}
is $k$, by the induction hypotheses, it is the lexicographically first chain between $x_1'$ and $y$.

We are going to find contradictions to each of the following possibilities.

        \begin{center}
        Case 1: $type(x_0,x_1)=c$, and $type(x_0,x_1')=c$,\\
        Case 2: $type(x_0,x_1)=d$, and $type(x_0,x_1')=d$, \\
        Case 3: $type(x_0,x_1)=d$, and $type(x_0,x_1')=c$,\\
        Case 4: $type(x_0,x_1)=c$, and $type(x_0,x_1')=d$,\\
        Case 5: $type(x_0,x_1)=r$, and $type(x_0,x_1')=r$,\\
        Case 6: $type(x_0,x_1)=d$, and $type(x_0,x_1')=r$, \\
        Case 7: $type(x_0,x_1)=r$, and $type(x_0,x_1')=d$,\\
        Case 8: $type(x_0,x_1)=r$, and $type(x_0,x_1')=c$,\\
        Case 9: $type(x_0,x_1)=c$, and $type(x_0,x_1')=r$,\\
        \end{center}

In each of these cases we will construct a partial involution $z$ such  
that $z$ covers $x_1'$ and $F((x_1',z))<F((x_1',x_2'))$.
Contradiction to the induction hypothesis.

{\em Case 1:} Proved in \cite{Incitti04}.

{\em Case 2:} $F(x_0,x_1)=(i,i)<F(x_0,x_1')=(j,j)$ with $i<j$. In $x_1'$ $(i,i)$ is a legal covering move. 
Hence we have our desired contradiction: $(j,j) \leq F((x_1',x_2'))\leq (i,i)$.

{\em Case 3:} $F(x_0,x_1)=(i,i)<F(x_0,x_1')=(j,k)$. There are two cases to consider: $i=j$ and $i<j$. 
If $i<j$ then we can reverse the order of the $d$ and $c$ move and get to the same contradiction as in \emph{Case 2}. 
If $i=j$ then $k \neq i+1$ since otherwise $(i,i)$ is not a possible move $x_0$. 
The $d$-cover moves $(i,i)$ to $(l,l)$ where $l<k$. 
But then $(i,l)$ is a legal move of $x_1'$ and $(i,l)<(i,k)$ which is a contradiction.

{\em Case 4:} $F(x_0,x_1)=(i,j)<F(x_0,x_1')=(k,k)$. There are two cases to be considered: $j=k$ and $j \neq k$.
If $j \neq k$ then $k \not \in [i,j]$ since otherwise $(i,j)$ is not a suitable rise. Hence $k>j$. 
But this means the two covering moves are interchangeable. 
We get to the same contradiction as in the preceding cases.
        
If $j=k$ then $(i,j)$ is a legal $r$-move of $x_1'$ with $(i,j)<(k,k)$. Contradiction.

{\em Case 5:} $F(x_0,x_1)=(i,j)<F(x_0,x_1')=(k,l)$. 
$k>i$ since there is at most one legal $r$-move of each element. We also have $j<l$ since otherwise either $(i,k)$ or $(i,l)$ is a 
suitable rise with a label less than $(i,j)$.
We have two cases to consider:
\begin{enumerate}
\item [(a)] {$i<j<k<l$}
\item [(b)] {$i<k<j<l$}
\end{enumerate}
In case (a), the two moves are interchangeable.
\\ In case (b), $(i,k)$ is a suitable rise of $x_0$ with $(i,k)<(i,j)$.

{\em Case 6:} $F(x_0,x_1)=(i,i)<F(x_0,x_1')=(j,k)$.
$j \neq i$ by construction. Hence $i<j<k$ and therefore $(j,k)$ does not influence the move $(i,i)$ and we derive a contradiction.

{\em Case 7:} $F(x_0,x_1)=(i,j)<F(x_0,x_1')=(k,k)$. $k>j$ since otherwise a suitable rise is removed by $(i,j)$. 
But then $(i,j)$ is a legal move of $x_1'$.

{\em Case 8:} $F(x_0,x_1)=(i,j)<F(x_0,x_1')=(k,l)$. Two cases need to be considered: $i<k$ and $i=k, i<j<l$. 
If $i<k$ then $j<k$ since otherwise a suitable rise would have been removed by $(i,j)$. 
But this means that $(i,j)$ is a legal move of $x_1'$.

If $i=k$ then the $c$-move corresponds to an $ef$, non-crossing $ee$, crossing $ee$ or a $ed$ rise. 
In each of these cases $(i,j)$ is a legal move of $x_1'$.

{\em Case 9:} $F(x_0,x_1)=(i,j)<F(x_0,x_1')=(k,l)$.
We have two cases to consider:
\begin{enumerate}
\item [(a)] {$i=k, i<j<l$}
\item [(b)] {$i<k$}
\end{enumerate} 
(a) does not occur because then the $r$-move removes a suitable rise, hence, it is not a covering relation.
\\ (b) Since we have $i<k<l$ and $i<j$, we consider $i<j<k<l, i<k<j<l \ \text{and} \ i<k<l<j$. 
In all these cases the $c$- and the $r$-moves are interchangeable. This ends the proofs of our claims. 

\end{proof}

Combining previous two propositions we obtain our first main result.
\begin{Theorem}
The poset of partial involutions is lexicographically shellable.
\end{Theorem}

\begin{figure}[htp]

\begin{center}

\begin{tikzpicture}[scale=.4]
media/.style={font={\footnotesize}},

\node at (0,0) (a) {$(1,2,3)$};
\node at (-8,5) (b1) {$(1,2,0)$};
\node at (0,5) (b2) {$(1,3,2)$};
\node at (8,5) (b3) {$(2,1,3)$};
\node at (-8,10) (c1) {$(1,0,3)$};
\node at (0,10) (c2) {$(2,1,0)$};
\node at (8,10) (c3) {$(3,2,1)$};
\node at (-8,15) (d1) {$(1,0,0)$};
\node at (0,15) (d2) {$(0,2,3)$};
\node at (8,15) (d3) {$(3,0,1)$};
\node at (-6,20) (e1) {$(0,2,0)$};
\node at (6,20) (e2) {$(0,3,2)$};
\node at (0,25) (f) {$(0,0,3)$};
\node at (0,30) (g) {$(0,0,0)$};

\node[red] at (0,2.5) {\scriptsize{(2,3)}};
\node[red] at (-4,2.5) {\scriptsize{(3,3)}};
\node[red] at (4,2.5) {\scriptsize{(1,2)}};
\node[red] at (8,7.5) {\scriptsize{(1,3)}};
\node[red] at (-8,7.5) {\scriptsize{(2,2)}};
\node[red] at (6,8.5) {\scriptsize{(1,2)}};
\node[red] at (2.5,8.5) {\scriptsize{(3,3)}};
\node[red] at (-6,8.5) {\scriptsize{(2,3)}};
\node[red] at (-2.5,8.5) {\scriptsize{(1,2)}};
\node[red] at (8,12.5) {\scriptsize{(2,2)}};
\node[red] at (-8,12.5) {\scriptsize{(3,3)}};
\node[red] at (-2,14) {\scriptsize{(1,1)}};
\node[red] at (0,14) {\scriptsize{(1,2)}};
\node[red] at (2,14) {\scriptsize{(1,3)}};
\node[red] at (4.5,14) {\scriptsize{(1,3)}};
\node[red] at (6.5,14) {\scriptsize{(2,3)}};
\node[red] at (7,17.5) {\scriptsize{(1,2)}};
\node[red] at (-7,17.5) {\scriptsize{(1,1)}};
\node[red] at (3,17.5) {\scriptsize{(2,3)}};
\node[red] at (-3,17.5) {\scriptsize{(3,3)}};
\node[red] at (-3,22.5) {\scriptsize{(2,2)}};
\node[red] at (3,22.5) {\scriptsize{(2,3)}};
\node[red] at (0,27.5) {\scriptsize{(3,3)}};


\draw[-](a) -- (b1);
\draw[-](a) -- (b2);
\draw[-](a) -- (b3);
\draw[-](b1) -- (c1);
\draw[-](b1) -- (c2);
\draw[-](b2) -- (c1);
\draw[-](b2) -- (c3);
\draw[-](b3) -- (c2);
\draw[-](b3) -- (c3);
\draw[-] (c1) -- (d1);
\draw[-](c1) -- (d2);
\draw[-](c1) -- (d3);
\draw[-](c2) -- (d2);
\draw[-](c2) -- (d3);
\draw[-](c3) -- (d2);
\draw[-](c3) -- (d3);
\draw[-] (d1) -- (e1);
\draw[-] (d2) -- (e1);
\draw[-](d2) -- (e2);
\draw[-](d3) -- (e2);
\draw[-] (e1) -- (f);
\draw[-] (e2) -- (f);
\draw[-] (f) -- (g);

\end{tikzpicture}
\end{center}

\caption{The $EL$-labeling of $P_3$.}
\label{F:Labeling of P3}

\end{figure}
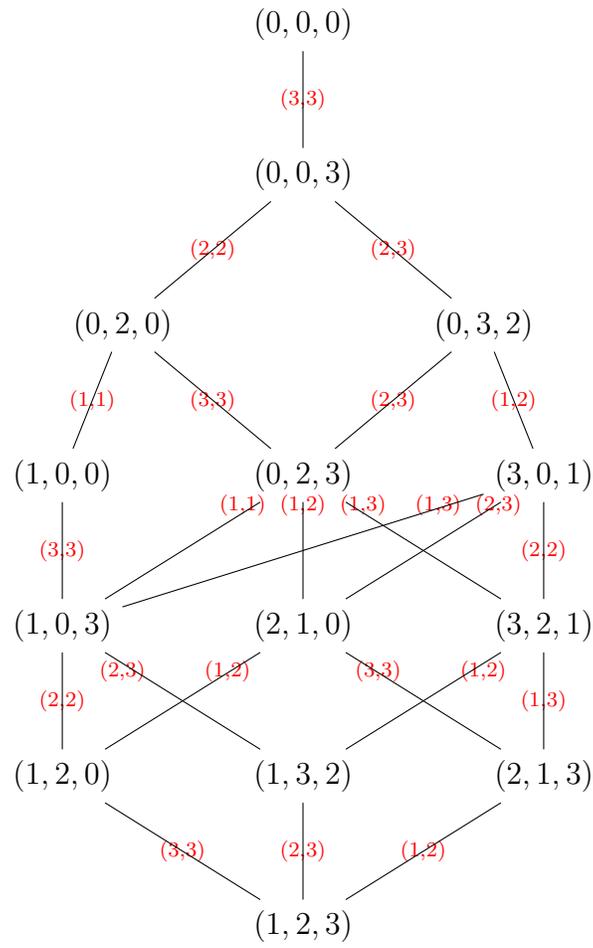

\section{\textbf{Eulerian intervals.}}\label{S:proof2}

In this section we prove Theorems \ref{T:Eulerian} and  \ref{T:union isomorphisms}.

There is a concrete way to compare two rooks given in one line notation in Bruhat-Chevalley-Renner ordering. For an integer valued vector $a=(a_1,\dots,a_n)\in \Z^n$, let
$\widetilde{a} = (a_{\alpha_1},\dots,a_{\alpha_n})$  be the
rearrangement of the entries $a_1,\dots,a_n$ of $a$ in a non-increasing
fashion;
\begin{equation*}
a_{\alpha_1} \geq a_{\alpha_2} \geq \cdots \geq a_{\alpha_n}.
\end{equation*}
The \textit{containment ordering}, ``$\leq_c$,'' on $\Z^n$ is then defined by
\begin{equation*}
a=(a_1,\dots,a_n) \leq_c b=(b_1,\dots,b_n) \iff a_{\alpha_j} \leq
b_{\alpha_j}\ \text{for all}\  j=1,\dots,n.
\end{equation*}
where $\widetilde{a} = (a_{\alpha_1},\dots,a_{\alpha_n})$, and
$\widetilde{b} = (b_{\alpha_1},\dots,b_{\alpha_n})$.
\begin{Example}
Let $x=(4,0,2,3,1)$, and let $y=(4,3,0,5,1)$. Then $x \leq_c y$, because
\begin{eqnarray*}
\widetilde{x}=(4,3,2,1,0) &\text{and}& \widetilde{y}=(5,4,3,1,0).
\end{eqnarray*}
\end{Example}
For $k\in [n]$, the \textit{$k$-th truncation} $a(k)$ of $a=(a_1,\dots,a_n)$ is defined to be
\begin{equation*}
a(k)=(a_1, a_2,\dots,a_k).
\end{equation*}
Let $v=(v_1,\dots, v_n)$ and $w=(w_1,\dots, w_n)$ be two rooks in $R_n$. 
It is shown in \cite{CR11} that 
\begin{equation*}
v \leq w \iff \widetilde{v(k)} \leq_c \widetilde{w(k)}\ \text{for all}\ k=1,\dots,n.
\end{equation*}
\begin{Example}
Let $x=(0,1,2,3,4)$, and let $y=(4,3,2,5,1)$. Then $x \leq y$, because
\begin{eqnarray*}
\widetilde{x(1)}=(0) &\leq_c& \widetilde{y(1)}=(4) ,\\
\widetilde{x(2)}=(1,0) &\leq_c& \widetilde{y(2)}=(4,3), \\
\widetilde{x(3)}=(2,1,0) &\leq_c& \widetilde{y(3)}=(4,3,2), \\
\widetilde{x(4)}=(3,2,1,0) &\leq_c& \widetilde{y(4)}=(5,4,3,2), \\
\widetilde{x(5)}=(4,3,2,1,0) &\leq_c& \widetilde{y(5)}=(5,4,3,2,1).
\end{eqnarray*}
\end{Example}

The next lemma, whose proof is omitted, shows that for two permutations $x$ and $y$ of $S_n$,
the inequality $x\leq y$ can be decided in $n-1$ steps.
\begin{Lemma}\label{L:step}
Let $x=(a_1,\dots,a_n)$ and $y=(b_1,\dots,b_n)$ be two permutations in $S_n$. 
Then $x\leq y$ if and only if
\begin{equation*}
\widetilde{x(k)} \leq_c \widetilde{y(k)}\ \text{for}\ k=1,\dots,n-1.
\end{equation*}
\end{Lemma}

We are ready to prove the first half of Theorem \ref{T:union isomorphisms}.
\begin{Proposition}\label{P:first half}
The union $(R_{n,n-1} \cup R_{n,n},\leq )$ is isomorphic to the poset $(S_{n+1},\leq )$. 
\end{Proposition} 

We depict the isomorphism between $S_4$ and $R_{3,3} \cup R_{3,2}$ in Figure \ref{F:Rooks3}.

\begin{proof}

Let $u$ and $w$ denote the rooks  $u=(0,1,2,\dots, n)$ and $w=(n,n-1,\dots, 2,1)$. Then $R_{n,n-1} \cup R_n = [u,w]$.

We define a map $\psi$ between $[v,w]$ and $S_{n+1}$ as follows. 
If $x=(a_1,\dots, a_n) \in [v,w]$, then 
\begin{eqnarray}\label{E:psi}
\psi(x) =  (a_1+1,a_2+1,\dots,a_n+1,a_x),
\end{eqnarray}
where $a_x$ is the unique element of the set
\begin{equation*}
[n+1] \setminus  \{a_1+1,a_2+1,\dots,a_n+1\}.
\end{equation*}

We have two immediate observations.
\begin{enumerate}
\item If $x$ is already a permutation (in $R_{n,n}$), then $a_x=1$.
\item $\psi$ is injective, hence by Lemma \ref{L:supporting}, it is bijective as well.
\end{enumerate}

Now,  let $x=(a_1,\dots,a_n)$ and $y=(b_1,\dots,b_n)$ be two elements in $[v,w]$ such that $x \leq y$. 
For the sake of brevity, denote the ``shifted'' sequence $(a_1+1,\dots,a_n+1)$ associated with $x$ by $x'$.
Since increasing each entry of $x$ and $y$ by 1 does not change the relative sizes of the entries of $x$ and $y$, we have
\begin{equation*}
x' \leq y'.
\end{equation*}
Recall that this is equivalent to saying that $\widetilde{x'(k)} \leq_c \widetilde{y'(k)} $ for all $k=1,\dots,n$.
Since, $x'$ is the $n$-th truncation $\psi(x)(n)$ of the permutation $\psi(x)$, the proof of the theorem is 
complete by considering Lemma \ref{L:step}. The converse statement $``\psi(x)\leq \psi(y) \implies x\leq y''$ 
follows from the same argument. Therefore, $\psi$ is a poset isomorphism. 
\end{proof}

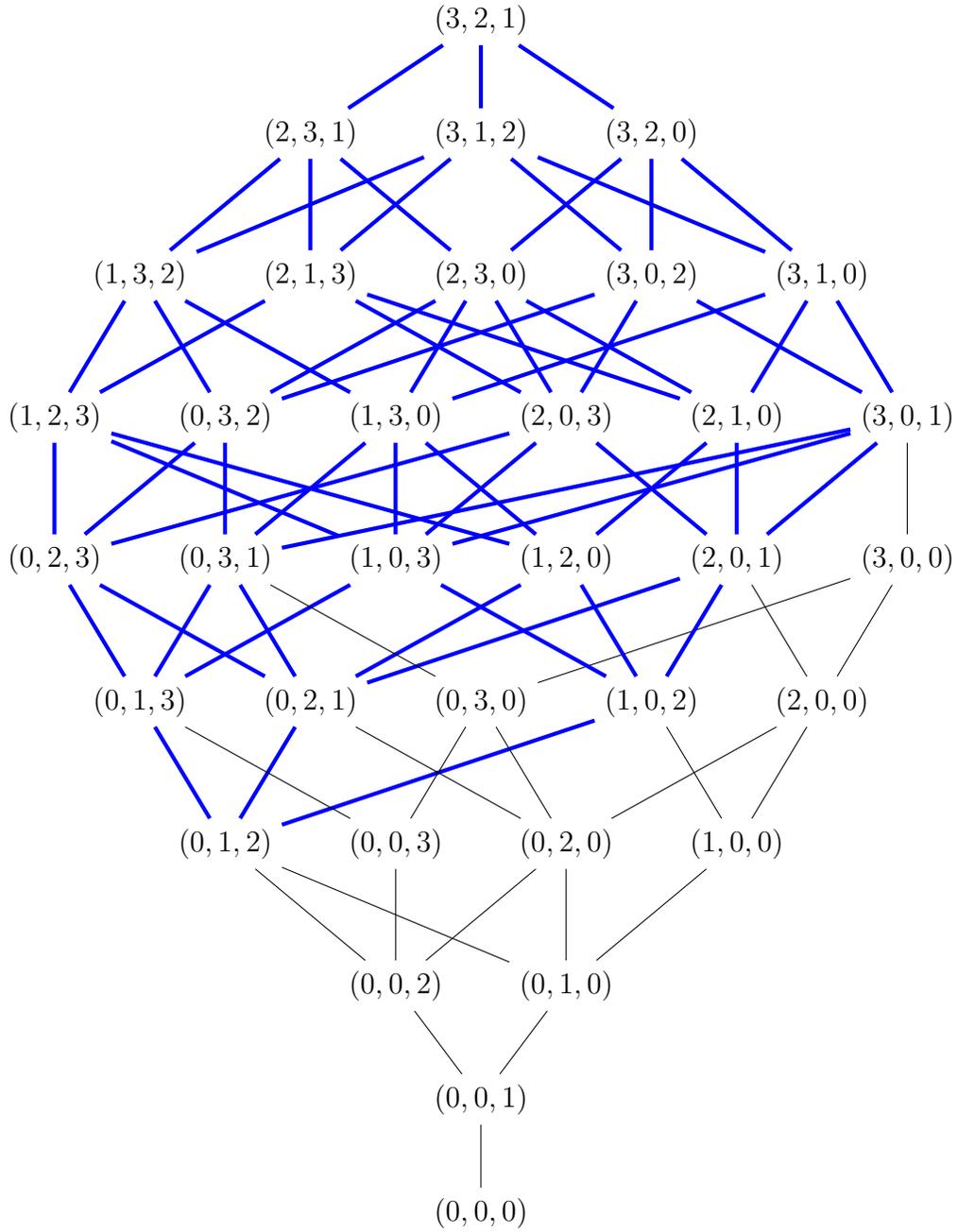
\begin{figure}[htp]

\begin{center}

\begin{tikzpicture}[scale=.4]
media/.style={font={\footnotesize}},

\node at (0,2) (a) {$(0,0,0)$};

\node at (0,6) (b) {$(0,0,1)$};

\node at (-3,10) (c1) {$(0,0,2)$};
\node at (3,10) (c2) {$(0,1,0)$};

\node at (-9,15) (d1) {$(0,1,2)$};
\node at (-3,15) (d2) {$(0,0,3)$};
\node at (3,15) (d3) {$(0,2,0)$};
\node at (9,15) (d4) {$(1,0,0)$};

\node at (-12,20) (e1) {$(0,1,3)$};
\node at (-6,20) (e2) {$(0,2,1)$};
\node at (0,20) (e3) {$(0,3,0)$};
\node at (6,20) (e4) {$(1,0,2)$};
\node at (12,20) (e5) {$(2,0,0)$};

\node at (-15,25) (f1) {$(0,2,3)$};
\node at (-9,25) (f2) {$(0,3,1)$};
\node at (-3,25) (f3) {$(1,0,3)$};
\node at (3,25) (f4) {$(1,2,0)$};
\node at (9,25) (f5) {$(2,0,1)$};
\node at (15,25) (f6) {$(3,0,0)$};

\node at (-15,30) (g1) {$(1,2,3)$};
\node at (-9,30) (g2) {$(0,3,2)$};
\node at (-3,30) (g3) {$(1,3,0)$};
\node at (3,30) (g4) {$(2,0,3)$};
\node at (9,30) (g5) {$(2,1,0)$};
\node at (15,30) (g6) {$(3,0,1)$};

\node at (-12,35) (h1) {$(1,3,2)$};
\node at (-6,35) (h2) {$(2,1,3)$};
\node at (0,35) (h3) {$(2,3,0)$};
\node at (6,35) (h4) {$(3,0,2)$};
\node at (12,35) (h5) {$(3,1,0)$};

\node at (-6,40) (i1) {$(2,3,1)$};
\node at (0,40) (i2) {$(3,1,2)$};
\node at (6,40) (i3) {$(3,2,0)$};

\node at (0,44) (j) {$(3,2,1)$};


\draw[-] (a) -- (b);

\draw[-] (b) -- (c1);
\draw[-] (b) -- (c2);

\draw[-] (c1) -- (d1);
\draw[-] (c1) -- (d2);
\draw[-] (c1) -- (d3);
\draw[-] (c2) -- (d1);
\draw[-] (c2) -- (d3);
\draw[-] (c2) -- (d4);

\draw[ultra thick,-,blue] (d1) -- (e1);
\draw[ultra thick,-,blue] (d1) -- (e2);
\draw[ultra thick,-,blue] (d1) -- (e4);
\draw[-] (d2) -- (e1);
\draw[-] (d2) -- (e3);
\draw[-] (d3) -- (e2);
\draw[-] (d3) -- (e3);
\draw[-] (d3) -- (e5);
\draw[-] (d4) -- (e4);
\draw[-] (d4) -- (e5);

\draw[ultra thick,-,blue] (e1) -- (f1);
\draw[ultra thick,-,blue] (e1) -- (f2);
\draw[ultra thick,-,blue] (e1) -- (f3);
\draw[ultra thick,-,blue] (e2) -- (f1);
\draw[ultra thick,-,blue] (e2) -- (f2);
\draw[ultra thick,-,blue] (e2) -- (f4);
\draw[ultra thick,-,blue] (e2) -- (f5);
\draw[-] (e3) -- (f2);
\draw[-] (e3) -- (f6);
\draw[ultra thick,-,blue](e4) -- (f3);
\draw[ultra thick,-,blue] (e4) -- (f4);
\draw[ultra thick,-,blue] (e4) -- (f5);
\draw[-] (e5) -- (f5);
\draw[-] (e5) -- (f6);

\draw[ultra thick,-,blue] (f1) -- (g1);
\draw[ultra thick,-,blue] (f1) -- (g2);
\draw[ultra thick,-,blue] (f1) -- (g4);
\draw[ultra thick,-,blue] (f2) -- (g2);
\draw[ultra thick,-,blue] (f2) -- (g3);
\draw[ultra thick,-,blue] (f2) -- (g6);
\draw[ultra thick,-,blue] (f3) -- (g1);
\draw[ultra thick,-,blue] (f3) -- (g3);
\draw[ultra thick,-,blue] (f3) -- (g4);
\draw[ultra thick,-,blue] (f3) -- (g6);
\draw[ultra thick,-,blue] (f4) -- (g1);
\draw[ultra thick,-,blue] (f4) -- (g3);
\draw[ultra thick,-,blue] (f4) -- (g5);
\draw[ultra thick,-,blue] (f5) -- (g4);
\draw[ultra thick,-,blue] (f5) -- (g5);
\draw[ultra thick,-,blue] (f5) -- (g6);
\draw[-] (f6) -- (g6);

\draw[ultra thick,-,blue] (g1) -- (h1);
\draw[ultra thick,-,blue] (g1) -- (h2);
\draw[ultra thick,-,blue] (g2) -- (h1);
\draw[ultra thick,-,blue] (g2) -- (h3);
\draw[ultra thick,-,blue] (g2) -- (h4);
\draw[ultra thick,-,blue] (g3) -- (h1);
\draw[ultra thick,-,blue] (g3) -- (h3);
\draw[ultra thick,-,blue] (g3) -- (h5);
\draw[ultra thick,-,blue] (g4) -- (h2);
\draw[ultra thick,-,blue] (g4) -- (h3);
\draw[ultra thick,-,blue] (g4) -- (h4);
\draw[ultra thick,-,blue] (g5) -- (h2);
\draw[ultra thick,-,blue] (g5) -- (h3);
\draw[ultra thick,-,blue] (g5) -- (h5);
\draw[ultra thick,-,blue] (g6) -- (h4);
\draw[ultra thick,-,blue] (g6) -- (h5);

\draw[ultra thick,-,blue] (h1) -- (i1);
\draw[ultra thick,-,blue] (h1) -- (i2);
\draw[ultra thick,-,blue] (h2) -- (i1);
\draw[ultra thick,-,blue] (h2) -- (i2);
\draw[ultra thick,-,blue] (h3) -- (i1);
\draw[ultra thick,-,blue] (h3) -- (i3);
\draw[ultra thick,-,blue] (h4) -- (i2);
\draw[ultra thick,-,blue] (h4) -- (i3);
\draw[ultra thick,-,blue] (h5) -- (i2);
\draw[ultra thick,-,blue] (h5) -- (i3);

\draw[ultra thick,-,blue] (i1) -- (j);
\draw[ultra thick,-,blue] (i2) -- (j);
\draw[ultra thick,-,blue] (i3) -- (j);

\end{tikzpicture}
\end{center}

\caption{$S_4$ in $(R_3,\leq)$.}
\label{F:Rooks3}

\end{figure}

Unfortunately, the map $\psi$ defined in (\ref{E:psi}) does not restrict to partial involutions nicely enough,
therefore, we need another order preserving injection in $P_{n,n-1} \cup P_n$ onto $I_{n+1}$.

Let $u=(0,n,n-1,\dots, 2)$ and let $\iota= (1,2,\dots,n)$. 
Observe that the rank-control matrix of $u$ is the smallest, and that the rank-control matrix of $\iota$ is the largest among all elements of $P_{n,n-1}\cup P_{n,n}$. Therefore, the union $P_{n,n-1} \cup P_{n,n}$ is 
the underlying set of the interval $[\iota,u]$ of $P_n$.

Let $x= (a_1,\dots, a_n) \in [\iota,u]$ be given in one-line notation. Then there are two cases: 
\begin{enumerate}
\item there is an $i\in [n]$ such that $a_i=0$,
\item $x$ is a permutation. 
\end{enumerate}
We start with the first case. 
If $a_i=0$ for some $i\in [n]$, then we define $b_i = n+1$ and for $j \in [n] \setminus \{i\}$
we set $b_j = a_j$. In addition, in this case, we define $b_{n+1}$ to be the unique element of the set
$\{0,1,\dots,n\}- \{ a_1,\dots,a_n\}$. If latter case, we set $b_j= a_j$ for $j=1,\dots, n$ and define $b_{n+1} = n+1$.
Finally, we define $\phi: [\iota,u] \rightarrow I_{n+1}$ by
\begin{align}\label{A:phi}
\phi( x) = (b_1,\dots, b_{n+1}).
\end{align}
For example, 
$$
x=\begin{pmatrix}  
0 & 1 &0 & 0 \\
1 & 0 & 0 & 0 \\
0 & 0 & 0 & 0 \\ 
0 & 0 & 0 & 1
\end{pmatrix}
\Rightarrow
\phi(x)=\begin{pmatrix}  
0 & 1 &0 & 0 &0 \\
1 & 0 & 0 & 0 &0\\
0 & 0 & 0 & 0 &1\\ 
0 & 0 & 0 & 1 &0\\
0 & 0 & 1 & 0 & 0
\end{pmatrix}.
$$

We are ready to prove the second half of Theorem \ref{T:union isomorphisms}.
\begin{Proposition}\label{P:second half}
The union $(P_{n,n-1} \cup P_{n,n},\leq )$ is isomorphic to the poset $(I_{n+1},\leq )$. 
\end{Proposition} 
We depict the isomorphism between $I_4$ and $P_{3,3} \cup P_{3,2}$ in Figure \ref{F:Involutions3}.

\begin{proof}
Let $\phi$ be defined as in (\ref{A:phi}). By its construction, $\phi$ is injective. 
Therefore, by Lemma \ref{L:supporting}, Part 2., it is enough to show that $\phi$ is order preserving. 

Let $x$ and $y$ be two elements in $[\iota,u]$ such that $x\leq y$. Then $R(y) \leq_R R(x)$.

Note that the upper-left $n\times n$ portion of the rank-control matrix of $\phi(x)$ is equal to $R(x)$. 
The same is true for $\phi(y)$ and $R(y)$. 

Let $R_{i,j}^{\phi(x)}$ denote the $(i,j)$-th entry of $R(\phi(x))$. 
Then, since $\phi(x)$ is a permutation in $I_{n+1}$, we have 
$$
R_{n+1,i}^{\phi(x)} = i\ \text{and}\ R_{j,n+1}^{\phi(x)}=j
$$ 
for all $i,j\in [n+1]$. The same is true for $R(\phi(y))$. Therefore, 
$$
R(\phi(y)) \leq_R R(\phi(x))
$$
and the proof is complete. 
\end{proof}

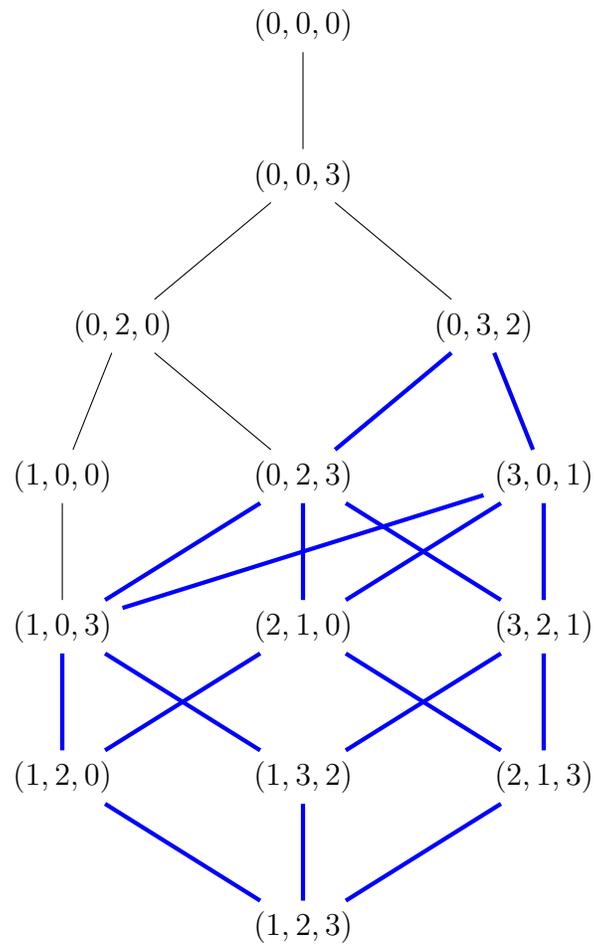
\begin{figure}[htp]

\begin{center}

\begin{tikzpicture}[scale=.4]
media/.style={font={\footnotesize}},

\node at (0,0) (a) {$(1,2,3)$};

\node at (-8,5) (b1) {$(1,2,0)$};
\node at (0,5) (b2) {$(1,3,2)$};
\node at (8,5) (b3) {$(2,1,3)$};
\node at (-8,10) (c1) {$(1,0,3)$};
\node at (0,10) (c2) {$(2,1,0)$};
\node at (8,10) (c3) {$(3,2,1)$};
\node at (-8,15) (d1) {$(1,0,0)$};
\node at (0,15) (d2) {$(0,2,3)$};
\node at (8,15) (d3) {$(3,0,1)$};
\node at (-6,20) (e1) {$(0,2,0)$};
\node at (6,20) (e2) {$(0,3,2)$};
\node at (0,25) (f) {$(0,0,3)$};
\node at (0,30) (g) {$(0,0,0)$};


\draw[ultra thick,-,blue] (a) -- (b1);
\draw[ultra thick,-,blue] (a) -- (b2);
\draw[ultra thick,-,blue] (a) -- (b3);
\draw[ultra thick,-,blue] (b1) -- (c1);
\draw[ultra thick,-,blue] (b1) -- (c2);
\draw[ultra thick,-,blue] (b2) -- (c1);
\draw[ultra thick,-,blue] (b2) -- (c3);
\draw[ultra thick,-,blue] (b3) -- (c2);
\draw[ultra thick,-,blue] (b3) -- (c3);
\draw[-] (c1) -- (d1);
\draw[ultra thick,-,blue] (c1) -- (d2);
\draw[ultra thick,-,blue] (c1) -- (d3);
\draw[ultra thick,-,blue] (c2) -- (d2);
\draw[ultra thick,-,blue] (c2) -- (d3);
\draw[ultra thick,-,blue] (c3) -- (d2);
\draw[ultra thick,-,blue] (c3) -- (d3);
\draw[-] (d1) -- (e1);
\draw[-] (d2) -- (e1);
\draw[ultra thick,-,blue] (d2) -- (e2);
\draw[ultra thick,-,blue] (d3) -- (e2);
\draw[-] (e1) -- (f);
\draw[-] (e2) -- (f);
\draw[-] (f) -- (g);

\end{tikzpicture}
\end{center}

\caption{$I_4$ in $(P_3,\leq)$.}
\label{F:Involutions3}

\end{figure}

It follows from Propositions \ref{P:first half} and \ref{P:second half} that Theorem \ref{T:union isomorphisms} is true;
$$
R_{n,n-1} \cup R_{n,n} \cong S_{n+1}\ \text{and}\ P_{n,n-1}\cup P_{n,n} \cong I_{n+1}. 
$$

Next we prove Theorem \ref{T:Eulerian}, which states that $R_{n,k}$ and $P_{n,k}$ are Eulerian if and only if $k=n$ or $k=n-1$.

First of all, $R_{n,n} \cong S_n$, and by Theorem \ref{T:union isomorphisms}, $R_{n,n-1}$ is isomorphic to an 
interval in $S_{n+1}$. Thus, both $R_{n,n}$ and $R_{n,n-1}$ are Eulerian. The same argument is true for both 
of the posets $P_{n,n}$ and $P_{n,n-1}$. 
Therefore, to finish the proof of Theorem \ref{T:Eulerian}, it is enough to show that, for $k\neq n,n-1$, $R_{n,k}$ and 
$P_{n,k}$ are not Eulerian. 
To this end, for $k\leq n-2$, let $v_k,v_k'$ and $v_k''$ denote the elements 
\begin{align*}
v_k &=(0,\dots,0,0,1,2,\dots, k),\\
v_k' &=(0,\dots,0,1,0,2,\dots, k),\\
v_k'' &=(0,\dots,1,0,0,2,\dots, k)
\end{align*}
in $R_{n,k}$. 
Then the interval $[v_k,v_k''] \subset R_{n,k}$ has exactly three elements $v_k,v_k',v_k''$, hence it cannot be Eulerian.

Similarly, for $k\leq n-2$, let $u_k,u_k'$ and $u_k''$ denote the elements 
\begin{align*}
u_k &=(1,2,\dots, k,0,\dots,0),\\
u_k' &=(1,2,\dots, k-1,0, k+1,0,\dots,0),\\
u_k'' &=(1,2,\dots,k-1,0,0,k+2,0,\dots, 0)
\end{align*}
in $I_{n,k}$. 
Then the interval $[u_k,u_k''] \subset P_{n,k}$ has exactly three elements $u_k,u_k',u_k''$, and therefore, it cannot be Eulerian. 
This finishes the proof of Theorem \ref{T:Eulerian}.

\bibliography{Partial_Involutions.bib}
\bibliographystyle{plain}
\end{document}